\documentclass[12pt]{amsart}
\usepackage{amssymb}
\usepackage{hyperref} 
\usepackage{mathrsfs}
\usepackage[all]{xy}
\usepackage{thm-restate}

\usepackage{verbatim}
\usepackage{comment}
\usepackage{geometry}
\usepackage{tikz,graphicx}
\tikzstyle{vertex}=[circle, draw, inner sep=0pt, minimum size=4pt]
\newcommand{\vertex}{\node[vertex]}

\definecolor{applegreen}{rgb}{0.55,0.71,0.0}

\definecolor{darkblue}{rgb}{0.0, 0.0, 0.7}

\newtheorem{theorem}{Theorem}[section]
\newtheorem{lemma}[theorem]{Lemma}
\newtheorem{proposition}[theorem]{Proposition}
\newtheorem{corollary}[theorem]{Corollary}
\theoremstyle{definition}
\newtheorem{definition}[theorem]{Definition}
\newtheorem{example}[theorem]{Example}

\newtheorem{remark}[theorem]{Remark}

\newcommand\ba{\mathbf{a}}

\newcommand\be{\mathbf{e}}

\newcommand\bbA{\mathbb{A}}

\newcommand\bbR{\mathbb{R}}

\newcommand\bbZ{\mathbb{Z}}
\newcommand\Z{\mathbb{Z}}

\newcommand\calC{\mathcal{C}}
\newcommand\calD{\mathcal{D}}

\newcommand\calF{\mathcal{F}}

\newcommand\calP{\mathcal{P}}
\newcommand\calQ{\mathcal{Q}}
\newcommand\calR{\mathcal{R}}
\newcommand\calS{\mathcal{S}}

\newcommand\calU{\mathcal{U}}

\newcommand\inedge{\mathrm{in}}
\newcommand\outedge{\mathrm{out}}

\renewcommand\IJ{I,\overline{J}}
\newcommand\hatG{\widehat{G}}
\newcommand\AIJ{A(I,\overline{J})}
\newcommand\GIJ{G(I,\overline{J})}
\newcommand\UIJ{\calU_{I,\overline{J}}}
\newcommand\precj{\mathrm{prec}(\overline{j})}
\newcommand\precJ{\mathrm{prec}(\overline{J})}

\DeclareMathOperator{\Cat}{Cat}
\DeclareMathOperator{\Nar}{Nar}

\title{On the subdivision algebra for the polytope $\UIJ$}
\author[von Bell]{Matias von Bell}
\address[von Bell]{Department of Mathematics\\
         University of Kentucky\\
}
\email{matias.vonbell@uky.edu}

\author[Yip]{Martha Yip}
\address[Yip]{Department of Mathematics\\
         University of Kentucky\\
}
\email{martha.yip@uky.edu}
\date{\today}

\thanks{Martha Yip is partially supported by Simons Collaboration Grant 429920.}

\begin{document}
\parskip=3pt

\begin{abstract}
The polytopes $\UIJ$ were introduced by Ceballos, Padrol, and Sarmiento to provide a geometric approach to the study of $(\IJ)$-Tamari lattices.
They observed a connection between certain $\UIJ$ and acyclic root polytopes, and wondered if M\'esz\'aros' subdivision algebra can be used to subdivide all $\UIJ$. 
We answer this in the affirmative from two perspectives, one using flow polytopes and the other using root polytopes. 
We show that $\UIJ$ is integrally equivalent to a flow polytope that can be subdivided using the subdivision algebra. 
Alternatively, we find a suitable projection of $\UIJ$ to an acyclic root polytope which allows subdivisions of the root polytope to be lifted back to $\calU_{\IJ}$.
As a consequence, this implies that subdivisions of $\UIJ$ can be obtained with the algebraic interpretation of using reduced forms of monomials in the subdivision algebra.
In addition, we show that the $(\IJ)$-Tamari complex can be obtained as a triangulated flow polytope.
\end{abstract}
\maketitle 

\section{Introduction} \label{sec.introduction}

The {\em Tamari lattice} is a certain partial order on the set of Dyck paths and has been extensively studied since its introduction by Tamari~\cite{Tamari}. 
Its Hasse diagram is well-known to be the $1$-skeleton of the {\em associahedron}, the secondary polytope of a convex polygon.
Pr\'eville-Ratelle and Viennot generalized this to the notion of the $\nu$-Tamari lattice, which is a poset on the set of lattice paths that lie weakly above a given lattice path $\nu$ from $(0,0)$ to $(a,b)$.
The classical Tamari lattice is the special case when $\nu$ is the staircase path from $(0,0)$ to $(n,n)$.

Motivated by the desire to provide a geometric approach to the study of $\nu$-Tamari lattices, Ceballos, Padrol and Sarmiento~\cite{CPS19} introduced the polytopes 
$$\calU_n = \mathrm{conv}\left\{ (\be_i, \be_{\overline{j}}) \mid i\in [n], \overline{j}\in[\overline{n}], i \prec \overline{j} \right\}\subseteq \bbR^{2n}$$
where $[n]=\{1,\ldots,n\} \footnote{In \cite{CPS19} the notation $[n]$ was used for convenience to denote $\{0,1,...,n\}$. In our setting, however, it is more convenient to use the more common definition in combinatorics excluding $0$.}$, $[\overline{n}]=\{\overline{1},\ldots,\overline{n}\}$, and $1\prec\overline{1}\prec \cdots \prec n \prec \overline{n}$.
This is a subpolytope of a product of simplices.
The faces of $\calU_n$ are described by 
$$\UIJ = \mathrm{conv}\left\{ (\be_i, \be_{\overline{j}}) \mid i\in I, \overline{j}\in\overline{J}, i \prec \overline{j} \right\}$$
for $I\subseteq[n]$ and $\overline{J}\subseteq[\overline{n}]$.
It was shown in~\cite{CPS19} that the polytope $\UIJ$ possesses a regular triangulation $\mathscr{A}_{\IJ}$ whose dual graph is the Hasse diagram of the $(\IJ)$-Tamari lattice.  For this reason, $\mathscr{A}_{\IJ}$ is called the {\em $(\IJ)$-associahedral triangulation}.

Ceballos et al. observed that polytopes which possess an associahedral triangulation have appeared in many different guises.
Of particular interest to us is the work of Gelfand, Graev and Postnikov~\cite{GGP}, who showed that the type $A_{n-1}$ full root polytope
$$\calP_{A_{n-1}} = \mathrm{conv}\left\{\mathbf{0}, \be_i-\be_j\mid 1\leq i < j\leq n\right\} $$
has an associahedral triangulation.
More generally, a type $A_{n-1}$ root polytope is the convex hull of $\mathbf{0}$ and a subset of the positive roots $\Phi_{A_{n-1}}^+ = \{\be_i-\be_j\mid 1\leq i < j\leq n\}$.
M\'esz\'aros~\cite{Mes11} studied a subclass of root polytopes $\calP(G)$ associated to a graph $G$ defined in the following way.
Let $G=(V,E)$ be an acyclic simple graph on the vertex set $V=[n]$ such that every edge is directed from the smaller vertex to the larger one.  The edge $e=(i,j)$ with $i<j$ is identified with the positive root $\alpha_e=\be_i-\be_j$, and the cone associated to $G$ is $\mathrm{cone}(G) = \{ \sum_{e\in E(G)} c_e \alpha_e \mid c_e\in \bbR_{\geq0}\}$.
Then the type $A_{n-1}$ root polytope associated to the acyclic graph $G$ is
$$\calP(G) = \calP_{A_{n-1}} \cap \mathrm{cone}(G) \subseteq \bbR^n. $$
This defines the class of {\em acyclic root polytopes}.
The special case of the full root polytope $\calP_{A_{n-1}}$ is the case when $G$ is the path graph on $n$ vertices. 
It was shown in ~\cite{Mes11} that acyclic root polytopes can be subdivided via reductions of a monomial $M_G$ in the {\em subdivision algebra} (see Section~\ref{sec.subdivisionalgebra}).

Under the linear map $\bbR^{2n}\rightarrow \bbR^n$ defined by $(\be_i,\mathbf{0}) \mapsto \be_i$ and $(\mathbf{0}, \be_{\overline{j}}) \mapsto -\be_j$, the polytope $\calU_n$ projects to the full root polytope $\calP_{A_{n-1}}$, and more generally the linear map sends the class of polytopes $\mathfrak{U}_n = \{\UIJ \mid I\subseteq[n], \overline{J}\subseteq[\overline{n}]\}$ to a class of polytopes which contains the acyclic root polytopes (see Example~\ref{ex:notARP} for an instance of a $\UIJ$ which does not project to an acyclic root polytope under this map).
Consequently, Ceballos et al.~\cite[Section 1.4]{CPS19} asked if it is possible to obtain subdivisions of the polytopes $\UIJ$ with the algebraic interpretation of using reductions of monomials in the subdivision algebra.
In the first part of this article, we show that this is indeed possible.

As explained by M\'esz\'aros~\cite[Section 4]{Mes16}, the subdivision algebras for acyclic root polytopes and a certain class of flow polytopes are the same.  Specifically, she showed that for each acyclic root polytope $\calP(G)$, there is a {\em flow polytope} $\calF_{\widetilde{G}}$ (on the fully augmented $\widetilde{G}$) which projects to a polytope $\calS(G)$ that is affinely equivalent to the acyclic root polytope.  

Our main result Theorem~\ref{thm.main}, proven in Section~\ref{sec:mainthm}, shows that for each polytope $\UIJ$, there exists an acyclic root polytope $\calP(G)$ and a flow polytope $\calF_{\hatG}$ (on the partially augmented $\hatG$) associated to the graph $G=\GIJ$ such that the following diagram commutes:
\[
\xymatrix{ 
\calF_{\hatG} \ar[d]_{\pi_1} \ar[r]^{\varphi_1} 
	& \UIJ \ar[d]^{\pi_2}\\
\calS(G)\ar[r]_{\varphi_2} & \calP(G)
}
\]
The maps $\varphi_1$ and $\varphi_2$ are integral equivalences, and the maps $\pi_1$ and $\pi_2$ are projections.  
As a consequence, this shows that the subdivision algebra for $\calP(G)$ and for $\calF_{\hatG}$ can be used to subdivide $\UIJ$.

In Section~\ref{sec.nuTamari}, we show that the triangulation of $\calF_{\hatG}$ obtained by reducing a monomial $M_G$ in the subdivision algebra in the {\em length reduction order} is a geometric realization of the $(\IJ)$-Tamari complex.
Relevant background on subdivision algebras, root polytopes, and flow polytopes is given in Section~\ref{sec.preliminaries}.

\section{Preliminaries} \label{sec.preliminaries}

\subsection{The subdivision algebra} \label{sec.subdivisionalgebra}
The subdivisions of acyclic root polytopes $\calP(G)$ (Section~\ref{sec.ARP}) and certain flow polytopes (Section~\ref{sec.FP}) can be encoded algebraically using M\'esz\'aros' subdivision algebra.
In this section, our brief exposition is based on the works~\cite{Mes11, Mes15, Mes16, MS20} where the theory and applications of the subdivision algebra have been extensively developed.

\begin{definition} 
\label{def:subdivAlgebra}
The {\em subdivision algebra} $\calS(\beta)$ is an associative commutative algebra over the ring of polynomials $\Z[\beta]$, generated by the elements $\{x_{ij} \mid 1\leq i < j \leq n\}$, subject to the relation
$x_{ij}x_{jk} = x_{ik}x_{ij} + x_{jk}x_{ik} + \beta x_{ik}$, if $1\leq i<j<k\leq n$. 
\end{definition}

Given a polynomial $p$ in $\calS(\beta)$, we say that $p'$ is a {\em reduction} of $p$ if $p'$ is obtained from $p$ by substituting a factor $x_{ij}x_{jk}$ of each monomial of $p$ divisible by $x_{ij}x_{jk}$ with $x_{ik}x_{ij} + x_{jk}x_{ik} + \beta x_{ik}$, with $1\leq i < j < k \leq n$. 
A consequence of a reduction is that $p'$ has two more monomials than $p$. A {\em reduced form} of a monomial $M$ in $\calS(\beta)$ is defined to be the final polynomial $p_m$ in a sequence $M=p_0,p_1,...,p_m$, where each polynomial is obtained from the previous polynomial via a reduction, and no reductions are possible in $p_m$.
Reduced forms of a monomial are not necessarily unique.

A reduction in $S(\beta)$ can be described in terms of graphs in the following way.
\begin{definition}
\label{def:reduction} 
A pair of edges $(i,j)$ and $(j,k)$ in a graph $G$ with vertex set $[n]$ are said to be {\em non-alternating} if $i< j< k$. 
A {\em reduction} on a pair of non-alternating edges $(i,j)$ and $(j,k)$ in $G$ produces the graphs $G_1$, $G_2$, and $G_3$ on the vertex set $[n]$ with the following three respective edge sets:
\begin{align*}
    E(G_1) &= E(G)\setminus \{(j,k)\} \cup \{(i,k)\},\\
    E(G_2) &= E(G)\setminus \{(i,j)\} \cup \{(i,k)\},\\
    E(G_3) &= E(G) \setminus \{(i,j),(j,k)\} \cup \{(i,k)\}.
\end{align*}
If a reduction can be performed on a pair of edges in $G$, we say that $G$ is {\em reducible}.
\end{definition}

A {\em reduction tree} $R_G$ of a graph $G$ is a rooted tree with nodes labeled by graphs that is constructed as follows. Begin with a tree consisting of only a root vertex $G$. 
If $G$ is reducible, perform a reduction on $G$ and add to the root the three children $G_1$, $G_2$ and $G_3$ resulting from the reduction as described in Definition~\ref{def:reduction}. Continue performing reductions on the reducible graphs at the leaves of the tree, adding the three corresponding children to the leaf at each iteration. 
When reducing a non-alternating pair $(i,j)$ and $(j,k)$ in a leaf of the tree, any other graphs containing these two edges must also be reduced at the same pair.
Continue growing the tree in this manner until the resulting ternary tree has no reducible graphs as its leaves. This tree is a {\em reduction tree} of $G$. See Figure~\ref{fig:reductionTree} for an example.

Sometimes it is convenient to only consider a {\em simple reduction tree} of $G$, where a {\em simple reduction} on a pair of non-alternating edges in $G$ produces the graphs $G_1$ and $G_2$ only. 
The simple reduction tree is constructed in the same way as the reduction tree, but using simple reductions at each step. 
The resulting tree is thus a binary tree. In Figure~\ref{fig:reductionTree}, removing the leaves indexed by monomials divisible by $\beta$ gives a simple reduction tree.

The leaves of a reduction tree $R_G$ with the same number of edges as the root $G$ are said to be {\em full-dimensional}. 
A simple reduction tree has only full-dimensional leaves.
The graphs which correspond with the full-dimensional leaves in $R_G$ are {\em maximal alternating graphs} on the vertex set of $G$.

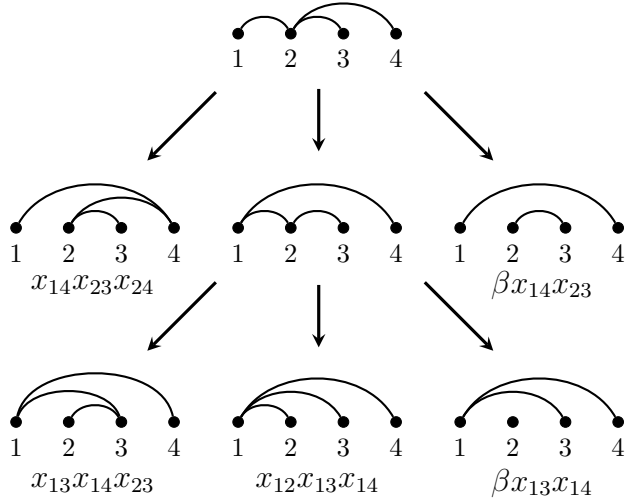
\begin{figure}
    \centering
\begin{tikzpicture}
\begin{scope}[scale=.7, xshift=0, yshift=0]
	\vertex[fill,label=below:\footnotesize{$1$}](a1) at (1,0) {};
	\vertex[fill,label=below:\footnotesize{$2$}](a2) at (2,0) {};
	\vertex[fill,label=below:\footnotesize{$3$}](a3) at (3,0) {};
	\vertex[fill,label=below:\footnotesize{$4$}](a4) at (4,0) {};	
	
	\draw[thick] (a1) to[out=60,in=120] (a2);
	\draw[thick] (a2) to[out=60,in=120] (a3);
	\draw[thick] (a2) to[out=60,in=120] (a4);
\end{scope}
\begin{scope}[scale=.7, xshift=-120, yshift=-105]
	\vertex[fill,label=below:\footnotesize{$1$}](a1) at (1,0) {};
	\vertex[fill,label=below:\footnotesize{$2$}](a2) at (2,0) {};
	\vertex[fill,label=below:\footnotesize{$3$}](a3) at (3,0) {};
	\vertex[fill,label=below:\footnotesize{$4$}](a4) at (4,0) {};	
	
	\draw[thick] (a2) to[out=60,in=120] (a3);
	\draw[thick] (a2) to[out=60,in=120] (a4);
	\draw[thick] (a1) to[out=60,in=120] (a4);
\end{scope}
\begin{scope}[scale=.7, xshift=0, yshift=-105]
	\vertex[fill,label=below:\footnotesize{$1$}](a1) at (1,0) {};
	\vertex[fill,label=below:\footnotesize{$2$}](a2) at (2,0) {};
	\vertex[fill,label=below:\footnotesize{$3$}](a3) at (3,0) {};
	\vertex[fill,label=below:\footnotesize{$4$}](a4) at (4,0) {};	

	\draw[thick] (a1) to[out=60,in=120] (a2);
	\draw[thick] (a2) to[out=60,in=120] (a3);
	\draw[thick] (a1) to[out=60,in=120] (a4);
\end{scope}
\begin{scope}[scale=.7, xshift=120, yshift=-105]
	\vertex[fill,label=below:\footnotesize{$1$}](a1) at (1,0) {};
	\vertex[fill,label=below:\footnotesize{$2$}](a2) at (2,0) {};
	\vertex[fill,label=below:\footnotesize{$3$}](a3) at (3,0) {};
	\vertex[fill,label=below:\footnotesize{$4$}](a4) at (4,0) {};	
	
	\draw[thick] (a2) to[out=60,in=120] (a3);
	\draw[thick] (a1) to[out=60,in=120] (a4);
\end{scope}
\begin{scope}[scale=.7, xshift=-120, yshift=-210]
	\vertex[fill,label=below:\footnotesize{$1$}](a1) at (1,0) {};
	\vertex[fill,label=below:\footnotesize{$2$}](a2) at (2,0) {};
	\vertex[fill,label=below:\footnotesize{$3$}](a3) at (3,0) {};
	\vertex[fill,label=below:\footnotesize{$4$}](a4) at (4,0) {};	
	
	\draw[thick] (a2) to[out=60,in=120] (a3);
	\draw[thick] (a1) .. controls (1.25,.75) and (2.75,.75) .. (a3);
	\draw[thick] (a1) .. controls (1.25,1.2) and (3.75,1.2) .. (a4);
\end{scope}
\begin{scope}[scale=.7, xshift=0, yshift=-210]
	\vertex[fill,label=below:\footnotesize{$1$}](a1) at (1,0) {};
	\vertex[fill,label=below:\footnotesize{$2$}](a2) at (2,0) {};
	\vertex[fill,label=below:\footnotesize{$3$}](a3) at (3,0) {};
	\vertex[fill,label=below:\footnotesize{$4$}](a4) at (4,0) {};	
	
	\draw[thick] (a1) to[out=60,in=120] (a2);
	\draw[thick] (a1) to[out=60,in=120] (a3);
	\draw[thick] (a1) to[out=60,in=120] (a4);
\end{scope}
\begin{scope}[scale=.7, xshift=120, yshift=-210]
	\vertex[fill,label=below:\footnotesize{$1$}](a1) at (1,0) {};
	\vertex[fill,label=below:\footnotesize{$2$}](a2) at (2,0) {};
	\vertex[fill,label=below:\footnotesize{$3$}](a3) at (3,0) {};
	\vertex[fill,label=below:\footnotesize{$4$}](a4) at (4,0) {};	
	
	\draw[thick] (a1) to[out=60,in=120] (a3);
	\draw[thick] (a1) to[out=60,in=120] (a4);
\end{scope}


\begin{scope}[scale=1.2, xshift=13, yshift=-15]
	\node[] (a1) at (0,0) {};
	\node[] (a2) at (-1,-1) {};
	\draw[-stealth, very thick] (a1)--(a2);
\end{scope}
\begin{scope}[scale=1.4, xshift=36, yshift=-12]
	\node[] (a1) at (0,0) {};
	\node[] (a2) at (0,-0.8) {};
	\draw[-stealth, very thick] (a1)--(a2);
\end{scope}
\begin{scope}[scale=1.2, xshift=72, yshift=-15]
	\node[] (a1) at (0,0) {};
	\node[] (a2) at (1,-1) {};
	\draw[-stealth, very thick] (a1)--(a2);
\end{scope}
\begin{scope}[scale=1.2, xshift=13, yshift=-75]
	\node[] (a1) at (0,0) {};
	\node[] (a2) at (-1,-1) {};
	\draw[-stealth, very thick] (a1)--(a2);
\end{scope}
\begin{scope}[scale=1.4, xshift=36, yshift=-65]
	\node[] (a1) at (0,0) {};
	\node[] (a2) at (0,-0.8) {};
	\draw[-stealth, very thick] (a1)--(a2);
\end{scope}
\begin{scope}[scale=1.2, xshift=72, yshift=-75]
	\node[] (a1) at (0,0) {};
	\node[] (a2) at (1,-1) {};
	\draw[-stealth, very thick] (a1)--(a2);
\end{scope}

\begin{scope}[scale=1, xshift=-35, yshift=-95]
	\node[] (a1) at (0,0) {$x_{14}x_{23}x_{24}$};
\end{scope}
\begin{scope}[scale=1, xshift=-35, yshift=-170]
	\node[] (a1) at (0,0) {$x_{13}x_{14}x_{23}$};
\end{scope}
\begin{scope}[scale=1, xshift=50, yshift=-170]
	\node[] (a1) at (0,0) {$x_{12}x_{13}x_{14}$};
\end{scope}
\begin{scope}[scale=1, xshift=135, yshift=-170]
	\node[] (a1) at (0,0) {$\beta x_{13}x_{14}$};
\end{scope}
\begin{scope}[scale=1, xshift=135, yshift=-95]
	\node[] (a1) at (0,0) {$\beta x_{14}x_{23}$};
\end{scope}

\end{tikzpicture}
    \caption{A reduction tree $R_G$ for the graph $G$ with $V(G) = [4]$ and edge set $E(G)= \{(1,2),(2,3),(2,4)\}$. The leaves of $R_G$ are encoded by the monomials in the reduced form of $M_G=x_{12}x_{23}x_{24}$ obtained in  Example~\ref{ex:reduction}.}
    \label{fig:reductionTree}
\end{figure}

When performing a reduction at a pair of non-alternating edges $a=(i,j)$ and $b=(j,k)$ in $G$, the added edge $(i,k)$ appearing in all three graphs $G_1$, $G_2$ and $G_3$ as in Definition~\ref{def:reduction} is treated as a formal sum $a+b(=b+a)$. Iterating this process beginning at the root lets us describe any edge in a node of $R_G$ as a sum of edges in $G$.

A graph $G$ gives rise to the monomial
$$M_G= \prod_{(a,b)\in E(G)} x_{ab}$$
in the subdivision algebra.
The relation $x_{ij}x_{jk} = x_{ik}x_{ij} + x_{jk}x_{ik} + \beta x_{ik}$ in Definition~\ref{def:subdivAlgebra} then encodes the reduction of $G$ at a pair of non-alternating edges $(i,j)$ and $(j,k)$. 
The monomials in the reduced form of $M_G$ give the leaves in the reduction tree. 

\begin{example}
\label{ex:reduction}
Let $G = ([4],\{(1,2),(2,3),(2,4)\})$ as depicted at the top of Figure~\ref{fig:reductionTree}. 
Then $M_G = x_{12}x_{23}x_{24}$, and a possible sequence of reductions of $M_G$ is as follows: 
\begin{align*}
    p_0 &= \mathbf{x}_{12}x_{23}\mathbf{x}_{24} \\
    p_1 &= \mathbf{x}_{12}\mathbf{x}_{23}x_{14} + x_{14}x_{23}x_{24}+ \beta x_{14}x_{23} \\
    p_2 &= x_{12}x_{13}x_{14} + x_{23}x_{13}x_{14} + x_{14}x_{23}x_{24} + \beta x_{13}x_{14} + \beta x_{14}x_{23} 
\end{align*}
The reduction in each step is performed on the pairs in boldface. 
This sequence of reductions corresponds with the reduction tree $R_G$ in Figure~\ref{fig:reductionTree}, with monomials of the reduced form $p_2$ of $M_G$ given by the five leaves of $R_G$, three of which are full-dimensional.  
\end{example}

\subsection{Acyclic root polytopes and the reduction lemma} \label{sec.ARP}
Recall from Section~\ref{sec.introduction} that the 
{\em full root polytope} is $\calP_{A_{n-1}} = \mathrm{conv}\{\mathbf{0}, \be_i-\be_j \mid 1\leq i < j \leq n \},$
and given an acyclic simple graph $G$ with vertex set $V\subseteq [n]$ such that every edge is directed from the smaller vertex to the larger one, the {\em acyclic root polytope associated to} $G$ is
$$\calP(G) = \calP_{A_{n-1}} \cap \mathrm{cone}(G) \subseteq \bbR^n,$$
where the edges $e=(i,j)\in E(G)$ are identified with the positive roots $\alpha_e = \be_i-\be_j$, and  $\mathrm{cone}(G) = \{ \sum_{e\in E(G)} c_e \alpha_e \mid c_e\in \bbR_{\geq0}\}$ is the cone associated to the edges of $G$.

Just as the full root polytope has a description as a convex hull of points, acyclic root polytopes have a similar alternative description.
Let $\Phi_G^+=\{ \alpha_e \mid e\in E(G)\}$ denote the set of positive roots of type $A_{n-1}$ corresponding to the edges of a graph $G$ on the vertex set $[n]$, and let $\overline{\Phi_{G}^+} = \Phi_{A_{n-1}}^+ \cap \mathrm{cone}(G)$ denote the set of positive roots of type $A_{n-1}$ contained in $\mathrm{cone}(G)$. 
Then we can also write
$$\calP(G) = \mathrm{conv} \{\mathbf{0}, \alpha \mid \alpha \in  \overline{\Phi_G^+}  \}.$$
The positive roots in $\overline{\Phi_{G}^+}$ determine the edges of a graph which may contain cycles.

A given acyclic root polytope can be generated from many different graphs, but amongst this class of graphs, there exists a unique acyclic representative. 
If the graph $G$ contains cycles and $\calP_{A_{n-1}}\cap \mathrm{cone}(G)$ happens to be an acyclic root polytope $\calP$, the {\em minimal graph} $\mathrm{min}(G)$ of $G$ is defined to be the unique acyclic representative in the class of graphs which generate $\calP$.  
The minimal graph has the property that no edge in $\mathrm{min}(G)$ can be written as a nonnegative linear combination of other edges in $\mathrm{min}(G)$.
For example, let $K_n$ and $P_n$ respectively denote the complete graph and path graph on the vertex set $[n]$.  
Then $\mathrm{min}(K_n)= P_n$ and $\calP(K_n) = \calP(P_n)$ is the full root polytope.

\begin{lemma}[{Reduction lemma for root polytopes~\cite[Proposition 1]{Mes11}, \cite[Lemma 4.1]{Mes16}}]
\label{lem:reductionLemmaARP}
Let $G=(V,E)$ be a graph with $V \subseteq [n]$ and $(i,j)$, $(j,k) \in E$ with $i<j<k$. For each $\ell\in[3]$, let $G_\ell$ be as in Definition~\ref{def:reduction}.
Then
$$\calP(G) = \calP(G_1) \cup \calP(G_2), 
\quad \calP(G_1) \cap \calP(G_2) = \calP(G_3),
\quad\hbox{and} \quad \calP(G_1)^\circ \cap \calP(G_2)^\circ = \emptyset, $$
where $\calP(G)$, $\calP(G_1)$, and $\calP(G_2)$ are of the same dimension, $\calP(G_3)$ is of one dimension less, and $\calP^\circ$ denotes the interior of $\calP$. 
\end{lemma}
Given an acyclic root polytope $\calP(G)$, M\'esz\'aros' reduction lemma explains how performing a reduction on $G$, or equivalently performing a reduction of the monomial $M_G$ in the subdivision algebra, is the same as dissecting $\calP(G)$ into two root polytopes $\calP(G_1)$ and $\calP(G_2)$ of the same dimension as $\calP(G)$, whose intersection $\calP(G_3)$ is a facet of each, and whose interiors are disjoint.
From this, we see that the leaves of a reduction tree $R_G$, or equivalently a reduced form of $M_G$, encodes a triangulation of $\calP(G)$.

\subsection{Flow polytopes and the reduction lemma} \label{sec.FP}

Let $G=(V,E)$ be a connected directed graph with vertex set $V\subseteq [n]$.
We assume that every edge is directed from the smaller vertex to the larger one.
In addition to this, we assume that the in-degree $\inedge_i$ of the vertex $i$ is positive if $i\in V$ and $2 \leq i \leq n$, and that the out-degree $\outedge_i$ is positive if $i\in V$ and $1 \leq i \leq n-1$. 
A directed graph with these properties will be called a {\em flow graph}. 
The {\em inner vertices} of a flow graph $G$ are the vertices that are not source or sink vertices, and edges whose end points are both inner vertices are {\em inner edges}.

Given a flow graph $G$ and a vector $\ba=(a_1,\ldots,a_{n-1},-\hbox{$\sum_{i=1}^{n-1}a_i$}) \in \bbZ^{n}$, an {\em $\ba$-flow on $G$}  is a tuple $(x_e)_{e\in E} \in \bbR^m_{\geq0}$ such that
$$\sum_{e \in \outedge(j)} x_e - \sum_{e \in \inedge(j)}  x_e = a_j$$
where $\inedge(j)$ and $\outedge(j)$ respectively denote the set of incoming and outgoing edges at vertex $j$, for $j=1,\ldots, n$. 
The {\em flow polytope of $G$ with net flow $\ba$} is the set $\calF_G(\ba)$ of $\ba$-flows on $G$.
In this paper we only consider flow polytopes with unit flow $\ba=\be_1 - \be_{n}$, and we will abbreviate the flow polytope of $G$ with unit flow as $\calF_G$.

The vertices of $\calF_G$ are the unit flows along maximal paths from vertex $1$ to vertex $n$. Such maximal paths are called {\em routes}. We denote the set of all routes in $G$ by $\calR(G)$.

\begin{definition} \label{defn:partially_augmented}
Given a graph $G$ with vertex set $V(G) \subseteq [n]$, its {\em augmented graph} is the connected graph $\widetilde{G}$ with vertex set $V(\widetilde{G}) = V(G) \cup \{s,t\} $, where $s<1<\cdots<n<t$, and edge set $E(\widetilde{G}) = E(G) \cup \{(s,i),(i,t)\mid i \in V(G)\}$. 
Removing from $E(\widetilde{G})$ the edges of the form $(s,i)$ where $i$ is a sink of $G$ and $(j,t)$ where $j$ is a source of $G$, we obtain the {\em partially augmented graph} $\hatG$.  
\end{definition}

Note that if $G$ is a flow graph, then directing the newly added edges of $\hatG$ from smaller to larger vertices also produces a flow graph.  

\begin{definition}
\label{def:F_G_i}
Let $G_N$ be a node of the reduction tree $R_G$. Define a map $\mu: E(G_N) \rightarrow \calR(\hatG)$ which takes an edge $(v_1,v_2) = e_{i_1} + \cdots + e_{i_\ell} \in G_N$, where $e_{i_j}\in E(G)$ for each $j\in [\ell]$, to the route $(s,v_1),e_{i_1},\ldots, e_{i_\ell},(v_2,t)$. Define the flow polytope $\calF_{\hatG_N}$ to be the convex hull of the vertices corresponding to the routes in the image of $\mu$ and the routes of the form $\{(s,i),(i,t)\}$ with $i\in [n]$. This definition of $\calF_{\hatG_N}$ is with respect to $G$ and is understood from context. 
\end{definition}

The following is a generalization of the reduction lemma for flow polytopes of fully augmented graphs due to Postnikov and Stanley.  
Also see~\cite[Lemma 2]{Mes15}, \cite[Proposition 2.3]{MS20}.

\begin{lemma}{(Reduction lemma for flow polytopes)}
\label{lem:reductionLemma}
Let $G$ be a flow graph with vertex set $V(G) \subseteq [n]$ and $(i,j)$, $(j,k) \in E(G)$ with $i<j<k$. For each $\ell\in[3]$, let $G_\ell$ be as in Definition~\ref{def:reduction}, and let $\calF_{\hatG_\ell}$ be as in Definition~\ref{def:F_G_i}. Then
$$\calF_{\hatG} = \calF_{\hatG_1} \bigcup  \calF_{\hatG_2},\quad \calF_{\hatG_1} \bigcap  \calF_{\hatG_2} = \calF_{\hatG_3}, \quad\text{and}\quad \calF_{\hatG_1}^\circ \bigcap  \calF_{\hatG_2}^\circ = \emptyset ,$$
where $\calF_{\hatG}$, $\calF_{\hatG_1}$, and $\calF_{\hatG_2}$ are of the same dimension, $\calF_{\hatG_3}$ is of one dimension less, and $\calP^\circ$ denotes the interior of $\calP$. 
\end{lemma}
\begin{proof}
From M\'esz\'aros' work we know that the reduction lemma holds for a fully
augmented graph $\widetilde{G}$. 
Suppose $(s,i)$ is an edge in $\widetilde{G}$, where $i$ is a sink of $G$. 
The only route in $\widetilde{G}$ using the edge $(s,i)$ is the route $(s,i),(i,t)$.
Furthermore, this holds in $\calF_{\widetilde{G}_\ell}$ for each $\ell\in [3]$, as no reduction can add an outgoing edge from $i$. 
Therefore, removing $(s,i)$ from $\widetilde{G}$ removes exactly the route $(s,i),(i,t)$ from $\calR(\widetilde{G}_1)$, $\calR(\widetilde{G}_2)$, and $\calR(\widetilde{G}_3)$. 
Now the subdivision of $\calF_{\widetilde{G}}$ into $\calF_{\widetilde{G}_1} \cup \calF_{\widetilde{G}_2}$ induces the subdivision of $\calF_{\widetilde{G}\setminus (s,i)}$ into $\calF_{\widetilde{G}_1\setminus (s,i)} \cup \calF_{\widetilde{G}_2\setminus (s,i)}$ satisfying the necessary conditions. 
The case of removing an edge $(j,t)$ where $j$ is a source of $G$ is similar.
Hence the reduction lemma for flow polytopes on partially augmented graphs follows.
\end{proof}

Given the partially augmented graph $\hatG$, Lemma~\ref{lem:reductionLemma} explains how performing a reduction on $G$, or equivalently performing a reduction of the monomial $M_G$ in the subdivision algebra, is the same as dissecting $\calF_{\hatG}$ into two flow polytopes $\calF_{\hatG_1}$ and $\calF_{\hatG_2}$ of the same dimension as $\calF_{\hatG}$, whose intersection $\calF_{\hatG_3}$ is a facet of each, and whose interiors are disjoint.
From this, we see that the leaves of a reduction tree $R_G$, or equivalently a reduced form of $M_G$, encodes a triangulation of $\calF_{\hatG}$.
In particular, the leaves of $R_G$ are the inner faces of the triangulation of $\calF_{\hatG}$.

\section{A subdivision algebra for \texorpdfstring{$\UIJ$}{}}
\label{sec:mainthm}

This section of the article is devoted to proving our main result.

\begin{restatable}{theorem}{main}
\label{thm.main}
Let $I\subseteq[n]$, $\overline{J}\subseteq[\overline{n}]$ be a valid pair. For the polytope $\UIJ$, there exists an acyclic root polytope $\calP(G)$ and a flow polytope $\calF_{\hatG}$ corresponding to a graph $G=\GIJ$ such that the following diagram commutes:
\[
\xymatrix{ 
\calF_{\hatG} \ar[d]_{\pi_1} \ar[r]^{\varphi_1} 
	& \UIJ \ar[d]^{\pi_2}\\
\calS(G)\ar[r]_{\varphi_2} & \calP(G)
}
\]
and $\varphi_1$ and $\varphi_2$ are integral equivalences.
\end{restatable}

As originally noted by Ceballos, Padrol, and Sarmiento in \cite{CPS19}, the projection $\bbR^{2n}\rightarrow \bbR^n$ that sends the subspace $\mathrm{span}\{(\be_i,\be_{\overline{i}}) \mid i\in [n]\} \subseteq \bbR^{2n}$ to zero sends the class of polytopes $\{\UIJ \mid I\subseteq[n], \overline{J}\subseteq[\overline{n}] \}$ to a class of polytopes which contains the acyclic root polytopes. 
When the projection of $\UIJ$ is an acyclic root polytope, the subdivision algebra for that root polytope can be applied to $\UIJ$ to obtain subdivisions, but if the projection of $\UIJ$ does not result in an acyclic root polytope (see Example~\ref{ex:notARP} for instance), this scheme fails.
However, we show how a modified projection $\pi_2$ (that depends on the choice of $\IJ$) sends $\UIJ$ to an acyclic root polytope $\calP(G)$ for a graph $G=\GIJ$ in a way that M\'esz\'aros' reduction lemma for acyclic root polytopes applies.
Additionally, we show that there is an integral equivalence $\varphi_1$ between a flow polytope $\calF_{\hatG}$ and $\UIJ$, so that the reduction lemma for flow polytopes of partially augmented graphs also applies to $\UIJ$.

From this theorem, we can conclude that reduced forms of monomials in the subdivision algebra for $\calP(G)$ and $\calF_{\hatG}$ is an algebraic tool for triangulating $\UIJ$.

\subsection{Constructing the graph \texorpdfstring{$\GIJ$}{} }
\label{sec.GIJ}

Each of the four polytopes in Theorem~\ref{thm.main} depends on a graph.  
We next explain the construction of each of these graphs.

Let $I\subseteq[n]$ and $\overline{J}\subseteq[\overline{n}]$, and let $\prec$ denote the total order $1 \prec \overline{1} \prec 2 \prec \overline{2} \prec \cdots \prec n \prec \overline{n}$ on $[n] \sqcup [\overline{n}]$. 
Denote by $\prec_{\IJ}$ the order induced on $I\sqcup \overline{J}$ by $\prec$. 
We say that the pair $(\IJ)$ is {\em valid} if $\min(I\sqcup \overline{J}) \in I$ and $\max(I\sqcup \overline{J}) \in \overline{J}$.
All pairs $(\IJ)$ are assumed to be valid unless otherwise stated. 
The restriction to valid pairs $(\IJ)$ does not change $\calU_{I,\overline{J}}$, but ensures that the graph $\GIJ$ that we will construct is connected, and guarantees that the polytopes in Theorem~\ref{thm.main} are non-empty.

To obtain the desired projection between $\UIJ$ and an acylic root polytope, we need a relabeling of certain elements in $\overline{J}$. 
Define the map $\mathrm{prec}: \overline{J} \rightarrow [n]$ as follows. 
If with respect to the order $\prec_{\IJ}$, the element $\overline{j} \in \overline{J}$ is not immediately preceded by an element in $I$, then $\mathrm{prec}(\overline{j})= j$.  Otherwise, 
$\mathrm{prec}(\overline{j})$ is defined to be this immediately preceding element that is in $I$.

Let $\AIJ$ be the graph with vertex set $I\sqcup \overline{J}$ and edge set $\{(i,\overline{j}) \mid i \prec_{\IJ} \overline{j}, i\in I, \overline{j}\in \overline{J} \}$.
If each $\overline{j}\in \overline{J}$ is identified with $\mathrm{prec}(\overline{j})$, then the identification partitions $I\sqcup \overline{J}$ into blocks of size one or two.
For any subgraph $H$ of $\AIJ$, define $\mathrm{prec}(H)$ to be the quotient graph of $H$ under this partition of its vertices.

We may identify the vertex set of $\mathrm{prec}(\AIJ)$ with $I \cup \mathrm{prec}(\overline{J})$, and its edge set is $\{ (i, \mathrm{prec}(\overline{j}) ) \mid i < \mathrm{prec}(\overline{j}), i\in I, \overline{j}\in \overline{J} \}$.
Observe that by construction, a vertex $v$ of $\mathrm{prec}(\AIJ)$ is a source if and only if $v\in I \backslash \mathrm{prec}(\overline{J})$, it is a sink if and only if $v\in \mathrm{prec}(\overline{J}) \backslash I$.  
Otherwise, a vertex $v\in I \cap \mathrm{prec}(\overline{J})$ has both incoming and outgoing edges.

Let $H$ be a simple graph on a linearly ordered vertex set whose edges are ordered from the smaller to the larger vertex.
Define the {\em minimal graph} $\min(H)$ to be the graph obtained from $H$ by removing every edge $(i,j)$ such that there is a directed path $i, i_1,\ldots, i_k, j$ in $H$ with $k \geq1$. This aligns with the previous definition of a minimal graph in Section~\ref{sec.ARP}. 

\begin{definition}\label{def:GIJ}
For a valid pair $(\IJ)$, the graph $\GIJ$ is the minimal graph of the graph $\mathrm{prec}(\AIJ)$. 
\end{definition}
We note that the linear order $\prec_{\IJ}$ on $I\sqcup \overline{J}$ induces an orientation on the edges of $\GIJ$ so that edges are directed from the smaller to the larger vertex.
Also, the tail of each edge is in $I$, while the head of each edge is in $\mathrm{prec}(\overline{J})$.

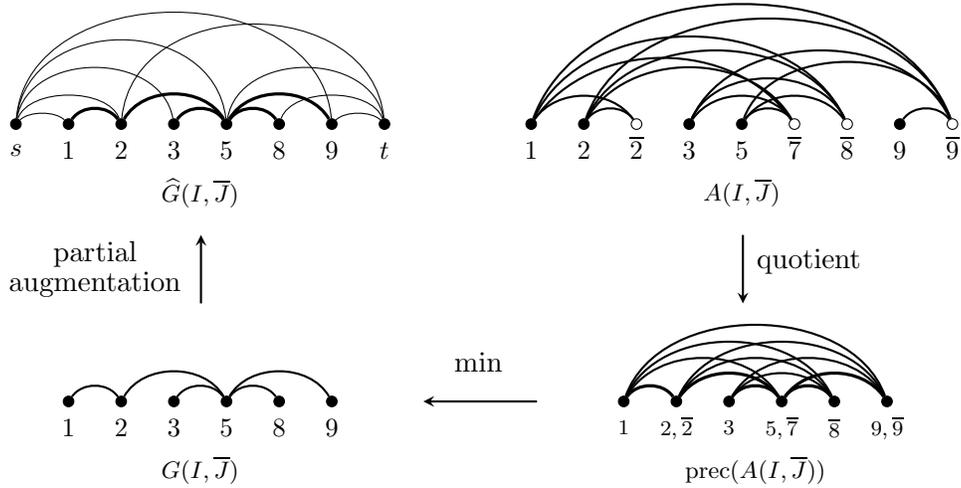
\begin{figure}
    \centering
\begin{tikzpicture}
\begin{scope}[scale=.7, xshift=250, yshift=150]
    \node() at (5,-1.3) {\scriptsize$\AIJ$};
	\vertex[fill](i1) at (1,0) {};
	\vertex[fill](i2) at (2,0) {};
	\vertex[](j2) at (3,0) {};
	\vertex[fill](i3) at (4,0) {};
	\vertex[fill](i5) at (5,0) {};
	\vertex[](j7) at (6,0) {};
	\vertex[](j8) at (7,0) {};
	\vertex[fill](i9) at (8,0) {};
	\vertex[](j9) at (9,0) {};
	\node() at (1,-0.5) {\footnotesize{$1$}};
	\node() at (2,-0.5) {\footnotesize{$2$}};
	\node() at (3,-0.45) {\footnotesize{$\overline{2}$}};
	\node() at (4,-0.5) {\footnotesize{$3$}};
	\node() at (5,-0.5) {\footnotesize{$5$}};
	\node() at (6,-0.45) {\footnotesize{$\overline{7}$}};
	\node() at (7,-0.45) {\footnotesize{$\overline{8}$}};
	\node() at (8,-0.5) {\footnotesize{$9$}};
	\node() at (9,-0.45) {\footnotesize{$\overline{9}$}};
	
	\draw[thick] (i1) to[out=55,in=125] (j2);%
	\draw[thick] (i1) to[out=55,in=125] (j7);%
	\draw[thick] (i1) to[out=65,in=115] (j8);
	\draw[thick] (i1) to[out=70,in=110] (j9);
	\draw[thick] (i2) to[out=55,in=125] (j2);%
	\draw[thick] (i2) to[out=60,in=120] (j7);
	\draw[thick] (i2) to[out=65,in=115] (j8);%
	\draw[thick] (i2) to[out=70,in=110] (j9);
	\draw[thick] (i3) to[out=55,in=125] (j7);%
	\draw[thick] (i3) to[out=65,in=115] (j8);%
	\draw[thick] (i3) to[out=65,in=115] (j9);%
	\draw[thick] (i5) to[out=55,in=125] (j7);%
	\draw[thick] (i5) to[out=55,in=125] (j8);%
	\draw[thick] (i5) to[out=70,in=110] (j9);
	\draw[thick] (i9) to[out=55,in=125] (j9);%
\end{scope}

\begin{scope}[scale=.7, xshift=300, yshift=0]
    \node() at (3.5,-1.3) {\scriptsize$\mathrm{prec}(\AIJ)$};
	\vertex[fill](a1) at (1,0) {};
	\vertex[fill](a2) at (2,0) {};
	\vertex[fill](a3) at (3,0) {};
	\vertex[fill](a5) at (4,0) {};
	\vertex[fill](a8) at (5,0) {};
	\vertex[fill](a9) at (6,0) {};
	\node() at (1,-.5) {\tiny{$1$}};
	\node() at (2,-.5) {\tiny{$2,\overline{2}$}};
	\node() at (3,-.5) {\tiny{$3$}};
	\node() at (4,-.5) {\tiny{$5,\overline{7}$}};
	\node() at (5,-.5) {\tiny{$\overline{8}$}};
	\node() at (6,-.5) {\tiny{$9,\overline{9}$}};
	\draw[very thick] (a1) to[out=55,in=125] (a2);%
	\draw[thick] (a1) to[out=60,in=120] (a5);%
	\draw[thick] (a1) to[out=65,in=115] (a8);%
	\draw[thick] (a1) to[out=70,in=110] (a9);
	\draw[very thick] (a2) to[out=55,in=125] (a5);%
	\draw[thick] (a2) to[out=60,in=120] (a8);%
	\draw[thick] (a2) to[out=65,in=115] (a9);%
	\draw[very thick] (a3) to[out=55,in=125] (a5);%
	\draw[thick] (a3) to[out=55,in=125] (a8);%
	\draw[thick] (a3) to[out=60,in=120] (a9);%
	\draw[very thick] (a5) to[out=55,in=125] (a8);%
	\draw[very thick] (a5) to[out=55,in=125] (a9);%
\end{scope}

\begin{scope}[scale=.7, xshift=0, yshift=0]
    \node() at (3.5,-1.3) {\scriptsize$\GIJ$};
	\vertex[fill](a1) at (1,0) {};
	\vertex[fill](a2) at (2,0) {};
	\vertex[fill](a3) at (3,0) {};
	\vertex[fill](a5) at (4,0) {};
	\vertex[fill](a8) at (5,0) {};
	\vertex[fill](a9) at (6,0) {};
	\node() at (1,-.5) {\footnotesize{$1$}};
	\node() at (2,-.5) {\footnotesize{$2$}};
	\node() at (3,-.5) {\footnotesize{$3$}};
	\node() at (4,-.5) {\footnotesize{$5$}};
	\node() at (5,-.5) {\footnotesize{$8$}};
	\node() at (6,-.5) {\footnotesize{$9$}};	
	\draw[thick] (a1) to[out=55,in=125] (a2);
	\draw[thick] (a2) to[out=60,in=120] (a5);
	\draw[thick] (a3) to[out=55,in=125] (a5);
	\draw[thick] (a5) to[out=55,in=125] (a8);
	\draw[thick] (a5) to[out=60,in=120] (a9);
\end{scope}

\begin{scope}[scale=.7, xshift=0, yshift=150]
    \node() at (3.5,-1.3) {\scriptsize$\hatG(\IJ)$};
	\vertex[fill](as) at (0,0) {};
	\vertex[fill](a1) at (1,0) {};
	\vertex[fill](a2) at (2,0) {};
	\vertex[fill](a3) at (3,0) {};
	\vertex[fill](a5) at (4,0) {};	
	\vertex[fill](a8) at (5,0) {};	
	\vertex[fill](a9) at (6,0) {};	
	\vertex[fill](at) at (7,0) {};
	\node() at (0,-.5) {\footnotesize{$s$}};
	\node() at (1,-.5) {\footnotesize{$1$}};
	\node() at (2,-.5) {\footnotesize{$2$}};
	\node() at (3,-.5) {\footnotesize{$3$}};
	\node() at (4,-.5) {\footnotesize{$5$}};
	\node() at (5,-.5) {\footnotesize{$8$}};
	\node() at (6,-.5) {\footnotesize{$9$}};
	\node() at (7,-.5) {\footnotesize{$t$}};
	
	\draw[very thick] (a1) to[out=55,in=125] (a2);
	\draw[very thick] (a2) to[out=60,in=120] (a5);
	\draw[very thick] (a3) to[out=55,in=125] (a5);
	\draw[very thick] (a5) to[out=55,in=125] (a8);
	\draw[very thick] (a5) to[out=60,in=120] (a9);

	\draw[] (as) .. controls (0.25,0.25) and (0.75,0.25) .. (a1);
	\draw[] (as) .. controls (0.25,0.7) and (1.75,0.7) .. (a2);
	\draw[] (as) .. controls (0.25,1.4) and (2.75,1.4) .. (a3);
	\draw[] (as) .. controls (0.25,2.1) and (3.75,2.1) .. (a5);
	\draw[] (as) .. controls (0.25,2.8) and (5.75,2.8) .. (a9);
	\draw[] (a2) .. controls (2.25,2.5) and (6.75,2.5) .. (at);
	\draw[] (a5) .. controls (4.25,1.4) and (6.75,1.4) .. (at);
	\draw[] (a8) .. controls (5.25,0.7) and (6.75,0.7) .. (at);
	\draw[] (a9) .. controls (6.25,0.25) and (6.75,0.25) .. (at);
\end{scope}

\begin{scope}[scale=1, xshift=70, yshift=33]
	\node[] (a1) at (0,0) {};
	\node[] (a2) at (0,1.2) {};
	\draw[-stealth, thick] (a1)--(a2);
\end{scope}

\begin{scope}[scale=1, xshift=275, yshift=33]
	\node[] (a1) at (0,0) {};
	\node[] (a2) at (0,1.2) {};
	\draw[stealth-, thick] (a1)--(a2);
\end{scope}

\begin{scope}[scale=1, xshift=150, yshift=0]
	\node[] (a1) at (0,0) {};
	\node[] (a2) at (1.8,0) {};
	\draw[stealth-, thick] (a1)--(a2);
\end{scope}

\begin{scope}[scale=1, xshift=300, yshift=53]
	\node[] (a1) at (0,0) {\small quotient};
\end{scope}

\begin{scope}[scale=1, xshift=175, yshift=15]
	\node[] (a1) at (0,0) {\small $\min$};
\end{scope}

\begin{scope}[scale=1, xshift=30, yshift=50]
    \node[] (a1) at (0,.2) {\small partial};
	\node[] (a1) at (0,-.2) {\small augmentation};
\end{scope}

\end{tikzpicture}
    \caption{Various graphs associated to the valid pair $I = \{1,2,3,5,9\}$ and $\overline{J} = \{\overline{2},\overline{7}, \overline{8},\overline{9} \}$. 
    The graph $\AIJ$ (top right), its quotient graph $\mathrm{prec}(\AIJ)$ (bottom right), its subsequent minimal graph $\GIJ$ (bottom left), and the partially augmented graph $\hatG(\IJ)$ (top left).
    See Remark~\ref{rem.fourgraphs} for an explanation of the roles played by these graphs in Theorem~\ref{thm.main}. } 
    \label{fig:G(I,J)}
\end{figure}

\begin{example} \label{ex:running}
Consider the valid pair $I= \{1,2,3,5,9 \}, \overline{J}=\{\bar{2},\bar{7}, \bar{8}, \bar{9}\}$.  
Then $1 \prec 2 \prec \bar{2} \prec 3 \prec 5 \prec \bar{7} \prec \bar{8} \prec 9 \prec\bar{9}$ with respect to the order $\prec_{\IJ}$, so $\mathrm{prec}(\overline{J}) = \{2,5,8,9 \}$. 
The graphs $\AIJ$ and $\mathrm{prec}(\AIJ)$ are shown on the right of Figure~\ref{fig:G(I,J)}.

The minimal graph $G=\GIJ$ has vertex set $I \cup \mathrm{prec}(\overline{J})=\{1,2,3,5,8,9 \}$. 
The partially augmented graph $\hatG=\hatG(\IJ)$ has vertex set $V(G) \cup \{ s,t\}$ and edge set $E(G) \cup \{(s,i) \mid i\in I \} \cup \{(\mathrm{prec}(\overline{j}, t) \mid \overline{j}\in \overline{J}\}$. 
These are shown on the left of Figure~\ref{fig:G(I,J)}.
\end{example}

\begin{remark}\label{rem.fourgraphs}
Looking ahead, the four graphs displayed in Figure~\ref{fig:G(I,J)} are central to the four polytopes of Theorem~\ref{thm.main} in the sense that the polytopes $\calS(G)$ and $\calF_{\hatG}$ can be seen as the convex hull of (sub)paths of the graphs $G=\GIJ$ and $\hatG$ respectively, while the polytopes $\UIJ$ and $\calP(G)$ can be seen as the convex hull of edges of the graphs $A=\AIJ$ and $\mathrm{prec}(A)$.
Perhaps this viewpoint helps to illuminate why these four polytopes have the `same' subdivision algebra.
\end{remark}

The following lemma is required for a key characterization of the routes in $\calF_{\hatG}$ used in the proof of Theorem~\ref{thm.main}.
\begin{lemma}
\label{lem:uniquepath}
For any two vertices $v$ and $w$ of $\GIJ$ with $v < w$ there exists at most one directed path from $v$ to $w$. 
\end{lemma}

\begin{proof}
Suppose there are two distinct directed paths $P$ and $Q$ from $v$ to $w$ in $\GIJ$. 
We can assume without loss of generality that $v$ is the first vertex after which $P$ and $Q$ differ. 
Then the first edges $(v,p_1)$ and $(v,q_1)$ of $P$ and $Q$ are different, and we may assume that $p_1 < q_1$.
Now, the path $P$ is of the form $v, p_1, \ldots, p_k, \ldots, w$ where $p_k$ is the largest vertex in $P$ which is smaller than $q_1$, and by the minimality of $\GIJ$ we have $p_k\neq q_1$.
Since $(v,q_1)$ and $(p_k, p_{k+1})$ are edges in $\GIJ$, then $p_k \in I$ and $q_1\in \mathrm{prec}(\overline{J})$.
But since $p_k < q_1$, then by construction the edge $(p_k,q_1)$ is in $\mathrm{prec}(A(\IJ))$, and therefore there is a path $P'$ in $\GIJ$ from $p_k$ to $q_1$.
However, the concatenation of the paths $v, p_1,\ldots, p_k$ and $P'$ is then a path from $v$ to $q_1$. This contradicts the minimality of $\GIJ$, as $\GIJ$ contains the edge $(v,q_1)$.
\end{proof}

We have just shown that as a directed graph, $\GIJ=\min \mathrm{prec}(\AIJ)$ is acyclic.  
However, Example~\ref{ex:notacyclic} shows that in general, it is possible to start with a graph and obtain a minimal graph that contains cycles (as an un-directed graph).
\begin{example} \label{ex:notacyclic}
Let $H$ be the graph on the vertex set $[4]$ with directed edges $(1,3)$, $(1,4)$, $(2,3)$, and $(2,4)$.  
Then $\min(H) = H$ is not acyclic as an un-directed graph.

This example illustrates why under the uniform projection map of Ceballos et al., certain $\UIJ$ do not project to acyclic root polytopes, as seen in Example~\ref{ex:notARP} for instance.
\end{example}

\begin{example} \label{ex:notARP}
Let $I = \{1,2\}$ and $\overline{J} = \{\overline{3},\overline{4}\}$ so that $\UIJ= \mathrm{conv}\{(\be_1,\be_{\overline{3}})$, $ (\be_1,\be_{\overline{4}})$, $(\be_2, \be_{\overline{3}})$, $(\be_2,\be_{\overline{4}})\}$. 
The map of Ceballos et al. which projects $\bbR^8 \rightarrow \bbR^4$ along the subspace spanned by $\{ (\be_i, \be_{\overline{i}}) \mid i=1,\ldots,4\}$ sends $\UIJ$ to the polytope $\calQ=\mathrm{conv}\{\be_1-\be_3, \be_1-\be_4, \be_2-\be_3, \be_2-\be_4\}$.
This is not an acyclic root polytope for the simple reason that it does not contain the origin. 
Aside from that, $\calQ$ also cannot be described as the convex hull of points $\overline{\Phi_G^+} = \Phi_{A_{3}}^+ \cap \mathrm{cone}(G)$ where $G$ is acyclic. 
If such a $G$ exists, it must contain the edges $(1,3)$, $(2,3)$, $(2,4)$ because they correspond to positive roots which cannot be expressed as positive linear combinations of other lower roots lying in $\calQ$, and if these edges are in $G$, then the acyclic $G$ cannot contain the edge $(1,4)$. 
Observe that the vertex $\be_1-\be_4 \notin \mathrm{cone}(G)$ for this $G$ (see Section~\ref{sec.ARP}). 
\end{example}

For the class of graphs $\mathrm{prec}(\AIJ)$, Lemma~\ref{lem:nonoriented_acyclic} shows that if we forget the orientation on the edges of $\GIJ = \min \mathrm{prec}(\AIJ)$, then $\GIJ$ remains acyclic as an undirected graph.
Hence we can define an acyclic root polytope on the graph $\GIJ$.

\begin{lemma}\label{lem:nonoriented_acyclic}
The graph $\GIJ$ is acyclic as an un-directed graph.
\end{lemma}
\begin{proof}
Suppose there is a cycle in $G=\GIJ$.
Let $v$ and $w$ be the smallest and largest vertices in the cycle respectively.
We can partition the cycle into two sequences of edges, each beginning at $v$ and ending at $w$.
By Lemma~\ref{lem:uniquepath} we know that these sequences cannot both form directed paths from $v$ to $w$. 
Thus at least one of the sequences is of the form $v=v_0, v_1, v_2, \ldots, v_r=w$ where $v_k > v_{k+1}$ for some $k\geq1$. 

Choose $k$ such that $v_{k+1}< v_k$.
Recall that if $(i,j)\in E(G)$, then $i\in I$ and $j\in \mathrm{prec}(\overline{J})$.
So by assumption, $v_{k+1}\in I$ and $v_k \in \mathrm{prec}(\overline{J})$.

Now consider two cases: either $v_k \in \mathrm{prec}(\overline{J})\backslash I$, or $v_k \in I \cap \mathrm{prec}(\overline{J})$.
First, suppose $v_k \in \mathrm{prec}(\overline{J})\backslash I$.
A consecutive string of vertices $z_1 < \cdots < z_p$ in $G$ which are in $\mathrm{prec}(\overline{J})\backslash I$ must be immediately preceded by a vertex $y \in I\cap \mathrm{prec}(\overline{J})$, because otherwise, $y$ immediately precedes $z_1$ in $\AIJ$ and the quotient map $\mathrm{prec}$ would have identified this pair of vertices in $\mathrm{prec}(\AIJ)$ and in $G$.

Note that if $u\in V(G)$ is a vertex that satisfies $v_{k+1} < u < v_k$, then $u\notin I \cap \mathrm{prec}(\overline{J})$, for if it is, then $(v_{k+1}, u)$ and $(u, v_k)$ are edges in $G$, and the positive root corresponding to $(v_{k+1}, v_k)$ can be written as a linear combination of the roots corresponding to $(v_{k+1}, u)$ and $(u, v_k)$, which is not possible because $G$ is a minimal graph.

Combining the last two observations, we conclude that $v_{k+1}\in I \cap \mathrm{prec}(\overline{J})$.
Then there exists a vertex $w$ in $\mathrm{prec}(\AIJ)$ where $w<v_{k+1}$ and $(w,v_{k+1})$ and $(w, v_k)$ are edges in $\mathrm{prec}(\AIJ)$.
Since $(v_{k+1}, v_k)$ is also an edge in $\mathrm{prec}(\AIJ)$, this contradicts the assumption that $G$ is a minimal graph.

Second, suppose $v_k \in I \cap \mathrm{prec}(\overline{J})$.
Then there exists a vertex $x$ in $\mathrm{prec}(\AIJ)$ where $v_k< x$ and $(v_{k+1},x)$ and $(v_k,x)$ are edges in $\mathrm{prec}(\AIJ)$.
Since $(v_{k+1}, v_k)$ is also an edge in $\mathrm{prec}(\AIJ)$, this contradicts the assumption that $G$ is a minimal graph.

Therefore, $G$ is acyclic as an undirected graph.
\end{proof}

\begin{lemma}\label{lem:Gconn}
The graph $\GIJ$ is connected.
\end{lemma}
\begin{proof}
First, we can see that the graph $\AIJ$ is connected by checking cases.
\begin{enumerate}
    \item Case $i\prec_{\IJ} \overline{j}$ with $i\in I$ and $\overline{j} \in \overline{J}$.
    The edge $(i,\overline{j})$ is in $\AIJ$ by definition.
    
    \item Case $i_1 \prec_{\IJ} i_2$ with $i_1, i_2\in I$. 
    For every $i\in I$, there is an edge $(i, \max{(I\sqcup\overline{J}}))$.  
    Thus $i_1, \max{(I\sqcup\overline{J}}), i_2$ is an (undirected) path in $\AIJ$ between $i_1$ and $i_2$.
    
    \item Case $\overline{j}_1\prec_{\IJ} \overline{j}_2$ with $\overline{j}_1,\overline{j}_2\in \overline{J}$.
    For every $\overline{j}\in \overline{J}$, there is an edge $(\min(I\sqcup\overline{J}), j)$.  
    Thus $j_1, \min(I\sqcup\overline{j}), j_2$ is an (undirected) path in $\AIJ$ between $j_1$ and $j_2$.
    
    \item Case $\overline{j} \prec_{\IJ} i$ with $\overline{j}\in \overline{J}$ and $i\in I$.
    Similar to the previous two cases, $\overline{j}, \min(I\sqcup \overline{J}), \max(I\sqcup \overline{J}), i$ is an undirected path in $\AIJ$ between $\overline{j}$ and $i$.
\end{enumerate}

Since $\AIJ$ is connected, then its quotient $\mathrm{prec}(\AIJ)$ is also connected.
Finally, an edge $e=(i,j)\in \mathrm{prec}(\AIJ)$ is removed from $\mathrm{prec}(\AIJ)$ to form $G= \min \mathrm{prec}(\AIJ)$ if and only if there exists a directed path from $i$ to $j$ in $\mathrm{prec}(\AIJ)$ of length at least two, so removing $e$ from $\mathrm{prec}(\AIJ)$ does not disconnect it, and it follows that $\GIJ$ is connected.
\end{proof}

Combining Lemmas~\ref{lem:nonoriented_acyclic} and~\ref{lem:Gconn}, we see that $\GIJ$ is in fact, a tree.

\subsection{Proof of the main theorem}\label{sec.main}
In this section we prove the main result, Theorem~\ref{thm.main}.
We begin by explaining how the preliminary material of Section~\ref{sec.preliminaries} applies in the context of the polytope $\UIJ$.
The reader may find it helpful to refer often to Figure~\ref{fig:G(I,J)}, as the four graphs $A=\AIJ$, $\mathrm{prec}(A)$, $G=\GIJ =\min \mathrm{prec}(A)$, and $\hatG$, play a central role in what follows.

For a valid pair $(\IJ)\in [n]\times [\overline{n}]$, suppose the graph $G$ has $m$ edges, so that the partially augmented graph $\hatG$ has $\widehat{m}=m+|I|+|\overline{J}|$ edges.
Let $\{\be_e \mid e \in E(G)\}$ be the standard orthonormal basis for $\bbR^m$ and let $\{\be_e \mid e \in E(\hatG)\}$ be the standard orthonormal basis for $\bbR^{\widehat{m}}$.
Recall from Section~\ref{sec.FP} that the vertices in a flow polytope correspond with the routes in its underlying graph. 
We have the flow polytope
$$\calF_{\hatG} = \mathrm{conv}\{ \hbox{routes $(s, v_1, \ldots, v_\ell, t)$ in $\hatG$} \} \subseteq \bbR^{\widehat{m}}.$$
The dimension of $\calF_{\hatG}$ is $|E(\hatG)|- |V(\hatG)|+1$.

Let $\pi_1 : \bbR^{\widehat{m}} \rightarrow \mathbb{R}^{m}$ be the linear map defined on the standard basis $\{ \be_{e} \mid e\in E(\hatG) \}$ for $\bbR^{\widehat{m}}$ by
$$
\pi_1(\be_e) = \begin{cases}
\mathbf{0}, &\hbox{if $e = (s,i)$ or $(j,t)$},\\
\be_{e}, &\hbox{otherwise}.
\end{cases}
$$
In other words, $\pi_1$ is the projection onto the coordinates associated with the inner edges of $\hatG$. 
Define the polytope $\calS(G)$ to be the image of $\calF_{\hatG}$ under $\pi_1$.
Note that since $(\IJ)$ is a valid pair, then $I\cap \mathrm{prec}(\overline{J}) \neq \emptyset$, and moreover every route in $\hatG$ of the form $(s,v,t)$ with $v\in I \cap \mathrm{prec}(\overline{J})$, projects to $\mathbf{0}$ under $\pi_1$.
The polytope $\calS(G)$ can equivalently be defined as
$$\calS(G) =\mathrm{conv}\{\text{paths } (v_1, \ldots, v_\ell) \text{ in } G \mid v_1\in I \text{ and } v_\ell \in \mathrm{prec}(\overline{J})  \} \subseteq \bbR^m,$$
and in this definition we include the empty path in $G$ so that $\mathbf{0}\in \calS(G)$.  

Next, recall that $A=\AIJ$ is the graph on $I\cup \overline{J}$ with the edge set $ \{(i,\overline{j}) \mid i \prec \overline{j} \} \subseteq  I \times \overline{J} $. 
It follows that we can express
$$\UIJ = \mathrm{conv}\{(\be_i,\be_{\overline{j}}) \mid (i,\overline{j}) \in E(A) \} \subseteq \bbR^{2n}.$$
The dimension of $\UIJ$ is $|I|+|\overline{J}|-2$.

By Lemma~\ref{lem:nonoriented_acyclic}, $G$ is acyclic as an undirected graph, so the acyclic root polytope defined on $G$ in Section~\ref{sec.ARP} is
$$\calP(G) = \mathrm{conv}\{ \mathbf{0}, \be_i -\be_j \mid (i,j) \in E(\mathrm{prec}(A))\} \subseteq \bbR^n,$$
where $\mathrm{prec}(A)$ is the quotient graph of $A$ defined in Section~\ref{sec.GIJ}.

Let $\pi_2 :\bbR^{2n}\rightarrow \bbR^n$ be the linear map defined on the standard basis of $\bbR^{2n}$ as follows. 
For $i=1,\ldots, n$, let $\pi_2((\be_i,\mathbf{0})) = \be_i$, and let
$$\pi_2((\mathbf{0}, \be_{\overline{j}})) =
\begin{cases}
-\be_{\mathrm{prec}(\overline{j})}, & \hbox{if } \overline{j}\in\overline{J},\\
-\be_j, & \hbox{if } \overline{j}\notin \overline{J}.
\end{cases}
$$
Let $\mathbb{A}_{\IJ}$ denote the affine span $\mathrm{aff}\{(\be_i,\be_{\overline{j}}) \mid (i,\overline{j}) \in E(A) \},$ so that $\UIJ\subseteq \bbA_{I,\overline{J}}$.
Now, the restriction $\pi_2:\mathbb{A}_{I,\overline{J}} \rightarrow \mathbb{R}^n$ gives $\pi_2(\be_{i}, \be_{\overline{j}} ) = \be_i - \be_{\mathrm{prec}(\overline{j})}$ for each $(i,\overline{j}) \in E(A)$, and the image of $\UIJ$ under $\pi_2$ is $\calP(G)$. 

Recall that $\bbR^{\widehat{m}} = \mathrm{span}\{ \be_{e} \mid e\in E(\hatG) \}$.
We define a linear map $\varphi_1: \bbR^{\widehat{m}} \rightarrow \bbR^{2n} $ on the standard basis of $\bbR^{\widehat{m}}$ by
$$\varphi_1(\be_{e}) = \begin{cases}
(\be_{i}, \mathbf{0}), & \hbox{if } e=(s,i), \\
(\mathbf{0}, \be_{\overline{j}}), & \hbox{if } e=(\mathrm{prec}(\overline{j}),t), \\
(\mathbf{0},\mathbf{0}), & \hbox{otherwise,}
\end{cases}$$
where $\mathbf{0}\in\bbR^n$.
Whereas $\pi_1$ was the projection onto the coordinates associated with the inner edges of $\hatG$, $\varphi_1$ is the map that sends the coordinates associated with the inner edges of $\hatG$ to zero.

Finally, recall that $\bbR^m = \mathrm{span}\{\be_e \mid e\in E(G) \}$.
We define a linear map $\varphi_2: \mathbb{R}^m \rightarrow \mathbb{R}^n$ on the standard basis of $\bbR^m$ by $\varphi_2(\be_{(i,j)}) = \be_i - \be_j$ for $(i,j) \in E(G)$.

M\'esz\'aros showed in~\cite[Theorem 4.4]{Mes15} that $\pi_1$ projects $\calF_{\hatG}$ onto $\calS(G)$, and $\calS(G)$ is affinely equivalent to $\calP(G)$ via $\varphi_2$. 
We will show that $\varphi_2$ is actually an integral equivalence; two integral polytopes $\calP\subseteq \bbR^a$ and $\calQ\subseteq\bbR^b$ are {\em integrally equivalent} if there exists an affine transformation $\varphi:\bbR^a\rightarrow\bbR^b$ whose restriction to $\calP$ is a bijection $\varphi:\calP\rightarrow \calQ$ that preserves the lattice.  That is, $\varphi$ is a bijection between $\bbZ^a \cap \mathrm{aff}(P)$ and $\bbZ^b \cap \mathrm{aff}(Q)$.
Integrally equivalent polytopes have the same Ehrhart polynomial, and hence the same normalized volume.

We are now ready to prove the main theorem.
\main*
\begin{proof}
We first check that $\varphi_1$ and $\varphi_2$ are integral equivalences.

Vertices in the flow polytope $\calF_{\hatG}$ correspond to routes in $\hatG$, and a route in $\hatG$ is of the form $R=(s, v_1, v_2, \ldots, v_\ell, t)$ for some $v_1=i\in I$, $v_\ell =\precj \in \precJ$ and $k \in V(G)$ for $k=1,\ldots, \ell$.
It follows from Lemma~\ref{lem:uniquepath} that there is a unique path from $v_1=i$ to $v_\ell=\precj$, so $R$ is completely determined by the two edges $(s,i)$ and $(\mathrm{prec}(\overline{j}),t)$.

The map $\varphi_1$ sends the routes of $\hatG$ to the vertices $(\be_{i}, \be_{\overline{j}})$ of $\UIJ$.
In particular, $\varphi_1$ is a bijection between the vertices of $\calF_{\hatG}$ and $\UIJ$, and it extends to a bijection between the polytopes as they are each the convex hull of their vertices.  
By Lemmas~\ref{lem:nonoriented_acyclic} and~\ref{lem:Gconn}, $G$ is a tree, so $|E(G)|=|V(G)|-1$, and hence
\begin{align*}
\dim \calF_{\hatG} 
&= |E(\hatG)| - |V(\hatG)| +1
= |E(G)| + |I| + |\overline{J}| - |V(G)| - 1
= |I| + |\overline{J}| -2 \\
&= \dim \UIJ.
\end{align*}
Therefore $\varphi_1$ maps $\calF_{\hatG}\subseteq [0,1]^{\widehat{m}}$ into $\UIJ\subseteq [0,1]^{2n}$ while preserving the dimension of the polytope, so it also preserves the respective lattices intersected with the affine spans of the polytopes.  Thus $\varphi_1$ is an integral equivalence.

M\'esz\'aros showed in~\cite[Theorem 4.4]{Mes15} that $\varphi_2$ is a linear map that restricts to a bijection between $\calS(G)$ and $\calP(G)$, so it remains to check that $\varphi_2$ preserves the lattice.
Note that the edges of $G$ form an orthonormal basis for $\mathbb{R}^{|E(G)|}$, and so $\dim(\calS(G)) = |E(G)|$.
Since the dimension of $\calP(G)$ is the number of edges in the minimal graph which generates it, then it also has dimension $|E(G)|$.
Therefore $\varphi_2$ is a dimension-preserving linear map between $[0,1]$-polytopes, and hence preserves the lattice. 

Finally, we check that the diagram commutes. 
By the linearity of the maps, it suffices to check that the square commutes for the vertices of $\calF_{\hatG}$. 
Let $x$ be a vertex of $\calF_{\hatG}$ that corresponds with the route $R=(s,v_1,\ldots, v_\ell, t)$ in $\hatG$ for some $v_1=i\in I$ and $v_\ell = \precj \in \precJ$.  
Then 
$$\pi_2(\varphi_1(x)) 
= \pi_2((\be_i,\be_{\overline{j}})) 
= \be_i-\be_{\precj}.$$ 
On the other hand, 
$$\varphi_2(\pi_1(x))
= \varphi_2\left(\sum_{k=1}^{\ell-1} \be_{(v_k,v_{k+1})} \right) 
= \sum_{k=1}^{\ell-1} \be_{v_k} - \be_{v_{k+1}}
= \be_i - \be_{\precj}.
$$
Thus the diagram commutes.
\end{proof}

Theorem~\ref{thm.main} is the key to obtaining subdivisions of $\calU_{I,\overline{J}}$ with the subdivision algebra. 
Since the flow polytope $\calF_{\hatG}$ is integrally equivalent to $\UIJ$ via $\varphi_1$, then the reductions of the monomial $M_G= \prod_{(i,j)\in E(G)} x_{ij}$ that encode subdivisions of $\calF_{\hatG}$ (Lemma~\ref{lem:reductionLemma}) also encode subdivisions of $\UIJ$ via $\varphi_1$, and we have the following corollary.

\begin{corollary}
\label{cor:subdivUIJ}
Reductions of the monomial $M_{G(I,\overline{J})}$ in the subdivision algebra encode subdivisions of $\calU_{I,\overline{J}}$. \qed
\end{corollary}

A second method to prove Corollary~\ref{cor:subdivUIJ} is through the associated acyclic root polytope, and this was the approach originally hinted at by Ceballos et al.
We explain this now.

A {\em cone point} of a triangulation (or simplicial complex) is a vertex which appears in every facet of the triangulation (or simplicial complex). 

\begin{lemma}\label{lem.conepoints1}
Let $(\IJ)$ be a valid pair and let $G=\GIJ$.
For $v\in I\cap \mathrm{prec}(\overline{J})$, the route $(s,v,t)$ in $\hatG$ is a cone point of any triangulation of $\calF_{\hatG}$ that is obtainable by the subdivision algebra.
\end{lemma}
\begin{proof}
Let $\mathscr{T}$ be a triangulation of $\calF_{\hatG}$ that is obtained by subdividing the monomial $M_G$ in some order.
Suppose the graph $H$ is a full-dimensional leaf in the reduction tree for $\mathscr{T}$.
Since $v\in I\cap \mathrm{prec}(\overline{J})$, then $(s,v,t)$ is a route in the partially augmented graph $\widehat{H}$ by definition.
This holds for every full-dimensional leaf $H$, hence $(s,v,t)$ is a vertex in every facet of $\mathscr{T}$.
\end{proof}

Via the integral equivalence $\varphi_1$, Lemma~\ref{lem.conepoints1} is equivalent to the following.
\begin{corollary}\label{lem.conepoints2}
Let $i\in I$ and $\overline{j}\in \overline{J}$ such that $\mathrm{prec}(\overline{j})=i \in I$.
The vertex $(\be_i, \be_{\overline{j}})$ is a cone point of any triangulation of $\UIJ$ that is obtainable by the subdivision algebra.
\qed
\end{corollary}

The vertices of $\UIJ$ that are mapped to $\mathbf{0}$ in $\calP(G)$ under the projection $\pi_2$ are precisely the vertices $(\be_i,\be_{\overline{j}})$ with $\mathrm{prec}(\overline{j})=i\in I$.  
By Corollary~\ref{lem.conepoints2}, these vertices are cone points of any triangulation of $\UIJ$ that is obtainable by the subdivision algebra, thus any triangulation of $\calP(G)$ obtained by subdividing $M_G$ induces a triangulation of $\UIJ$ by adding back the cone points $\{(\be_i,\be_{\overline{j}})\mid \mathrm{prec}(\overline{j}) =i\in I \}$.

We conclude this section with an example showing how to simultaneously obtain a triangulation of $\UIJ$, $\calF_{\hatG}$, and $\calP(G)$ via the subdivision algebra.
In this example, we reduce the monomial $M_G$ according to the length reduction order, defined as follows.

\begin{definition}\label{def:length}
The {\em length} of an edge $(i,j)$ in $G$ is $j-i$.
At a vertex $j$ of $G$, a pair of non-alternating edges $(i,j)$ and $(j,k)$ where $i<j<k$ form a {\em longest pair} at $j$ if there is no edge $(i',j)$ with $i'<i$ and no edge $(j,k')$ with $k<k'$.
The {\em length reduction order} is the reduction order obtained by reducing the longest pair $\{(i,j), (j,k) \}$ with minimal $j$ at each reduction step.
\end{definition}
Figure~\ref{fig:reductionTree} is an example of a reduction tree obtained using the length reduction order.

\begin{example}\label{ex:runningend} 
We finish the running example for the valid pair $I= \{1,2,3,5,9 \}, \overline{J}=\{\bar{2},\bar{7}, \bar{8}, \bar{9}\}$ from Example~\ref{ex:running}.
The graph $G=\GIJ$ is pictured in the bottom left of Figure~\ref{fig:G(I,J)}.
The flow polytope $\calF_{\hatG}\subseteq \bbR^{14}$ and the polytope $\UIJ\subseteq \bbR^{18}$ are each $7$-dimensional and contains $15$ vertices.
The polytope $\calS(G)\subseteq \bbR^5$ and the acyclic root polytope $\calP(G)\subseteq \bbR^9$ are each $5$-dimensional and contains $12$ vertices.

Reducing $M_{G}=x_{12}x_{25}x_{35}x_{58}x_{59}$ according to the length reduction order yields $16$ highest degree terms which correspond to the $16$ facets in the associated triangulations of $\calP(G)$, $\calF_{\hatG}$, and $\UIJ$.
One such highest degree monomial is $M=x_{12}x_{19}x_{38}x_{39}x_{58}$.
It encodes the maximal alternating non-crossing graph $G_M$ in Figure~\ref{fig:running}, which corresponds to the facet $\mathscr{F}$ that is the convex hull of $\mathbf{0}$, $\be_1-\be_2$, $\be_1-\be_9$, $\be_3-\be_8$, $\be_3-\be_9$, and $\be_5-\be_8$ in the triangulation of the acyclic root polytope $\calP(G)$.

In the context of the triangulation of the flow polytope $\calF_{\hatG}$, the monomial $M
$ corresponds to the facet that is the convex hull of the eight routes in $\widehat{G}_M$, shown on the left of the following table.
Under the integral equivalence map $\varphi_1$, this facet in the triangulation of $\calF_{\hatG}$ maps to the facet in the triangulation of $\UIJ$ consisting of the convex hull of the eight vertices on the right of the following table.
\begin{eqnarray*}
(s,9,t) & \leftrightarrow &(\be_9, \be_{\overline{9}})\\
(s,5,8,t) & \leftrightarrow &(\be_5, \be_{\overline{8}})\\
(s,5,t) & \leftrightarrow &(\be_5, \be_{\overline{7}})\\
(s,3,5,9,t) & \leftrightarrow &(\be_3, \be_{\overline{9}})\\
(s,3,5,8,t) & \leftrightarrow &(\be_3, \be_{\overline{8}})\\
(s,2,t) & \leftrightarrow &(\be_2, \be_{\overline{2}})\\
(s,1,2,5,9,t) & \leftrightarrow &(\be_1, \be_{\overline{9}})\\
(s,1,2,t) & \leftrightarrow &(\be_1, \be_{\overline{2}})
\end{eqnarray*}

Observe that $I\cap \mathrm{prec}(\overline{J}) = \{2,5,9\}$. 
The routes $(s,2,t)$, $(s,5,t)$, and $(s,9,t)$ in $\hatG$ are the vertices of $\calF_{\hatG}$ that are sent to $\mathbf{0}$ in $\calP(G)$ by the maps $\pi_2 \circ \varphi_1$ and $\varphi_2\circ\pi_1$, and by Lemma~\ref{lem.conepoints1}, these three routes are cone points in any triangulation of $\calF_{\hatG}$ obtainable by the subdivision algebra. 
Their counterparts in $\UIJ$ are the vertices $(\be_2, \be_{\overline{2}})$, $(\be_5, \be_{\overline{7}})$, and $(\be_9, \be_{\overline{9}})$.
Thus, we see how the facet $\mathscr{F}$ in the triangulation of $\calP(G)$ induces the corresponding facet of $\UIJ$: the five nonzero vertices in $\mathscr{F}$ correspond with the vectors $(\be_1, \be_{\overline{2}})$, $(\be_1, \be_{\overline{9}})$, $(\be_3,\be_{\overline{8}})$, $(\be_3, \be_{\overline{9}})$, $(\be_5,\be_{\overline{8}})$, and adding back the three cone points $(\be_2, \be_{\overline{2}})$, $(\be_5, \be_{\overline{7}})$, and $(\be_9, \be_{\overline{9}})$ recovers the corresponding facet in the triangulation of $\UIJ$.

Finally, we note that the eight vertices forming a facet in the triangulation of $\UIJ$ correspond to the arcs of what is known as an $(\IJ)$-tree (see the graph $T$ in Figure~\ref{fig:running}).  
We discuss the importance of the connection between $(\IJ)$-trees and full-dimensional leaves in a reduction tree in the next section.
\end{example}

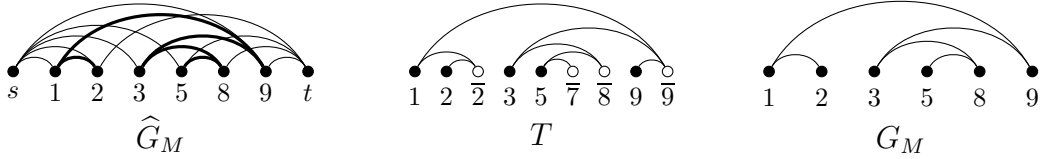
\begin{figure}
    \centering
\begin{tikzpicture}[scale=0.7]
\begin{scope}[xshift = 0, yshift = 0, scale=0.8]
	\node() at (3.5, -1.5) {$\widehat{G}_M$};
	
	\vertex[fill](s) at (0,0) {};
	\vertex[fill](1) at (1,0) {};
	\vertex[fill](2) at (2,0) {};
	\vertex[fill](3) at (3,0) {};
	\vertex[fill](5) at (4,0) {};
	\vertex[fill](8) at (5,0) {};
	\vertex[fill](9) at (6,0) {};
	\vertex[fill](t) at (7,0) {};	
    \node() at (0,-0.5) {\footnotesize{$s$}};
    \node() at (1,-0.5) {\footnotesize{$1$}};
    \node() at (2,-0.5) {\footnotesize{$2$}};
    \node() at (3,-0.5) {\footnotesize{$3$}};
    \node() at (4,-0.5) {\footnotesize{$5$}};
    \node() at (5,-0.5) {\footnotesize{$8$}};
    \node() at (6,-0.5) {\footnotesize{$9$}};
    \node() at (7,-0.5) {\footnotesize{$t$}};    

	\draw[very thick] (1) to [bend left=60] (9);
    \draw[very thick] (1) to [bend left=60] (2);
    \draw[very thick] (3) to [bend left=60] (9);  
    \draw[very thick] (3) to [bend left=60] (8);    
    \draw[very thick] (5) to [bend left=60] (8);
    
    \draw[] (s) to [bend left=60] (1);
    \draw[] (s) to [bend left=60] (2);
    \draw[] (s) to [bend left=60] (3);
    \draw[] (s) to [bend left=60] (5);
    \draw[] (s) to [bend left=60] (9);
    \draw[] (2) to [bend left=60] (t);
    \draw[] (5) to [bend left=60] (t);
    \draw[] (8) to [bend left=60] (t);
    \draw[] (9) to [bend left=60] (t);    
\end{scope}
\begin{scope}[xshift = 200, yshift = 0, scale=0.6]
	\node() at (5, -2) {$T$};	
	\vertex[fill](1) at (1,0) {};
	\vertex[fill](2) at (2,0) {};
	\vertex[fill](3) at (4,0) {};
	\vertex[fill](5) at (5,0) {};
	\vertex[fill](9) at (8,0) {};		
	\vertex[](2bar) at (3,0) {};
	\vertex[](7bar) at (6,0) {};
	\vertex[](8bar) at (7,0) {};
	\vertex[](9bar) at (9,0) {};
	
    \node() at (1,-0.8) {\footnotesize{$1$}};
    \node() at (2,-0.8) {\footnotesize{$2$}};
    \node() at (3,-0.7) {\footnotesize{$\overline{2}$}};
    \node() at (4,-0.8) {\footnotesize{$3$}};
    \node() at (5,-0.8) {\footnotesize{$5$}};
    \node() at (6,-0.7) {\footnotesize{$\overline{7}$}};
    \node() at (7,-0.7) {\footnotesize{$\overline{8}$}};
    \node() at (8,-0.8) {\footnotesize{$9$}};
    \node() at (9,-0.7) {\footnotesize{$\overline{9}$}};    
    
	\draw[] (1) to [bend left=60] (2bar);
    \draw[] (1) to [bend left=60] (9bar);
    \draw[] (2) to [bend left=60] (2bar);
    \draw[] (3) to [bend left=60] (8bar);
    \draw[] (3) to [bend left=60] (9bar);  
    \draw[] (5) to [bend left=65] (7bar);  
    \draw[] (5) to [bend left=60] (8bar);    
    \draw[] (9) to [bend left=60] (9bar);
\end{scope}
\begin{scope}[xshift = 380, yshift = 0]
	\node() at (3.5, -1.3) {$G_M$};
	\vertex[fill](1) at (1,0) {};
	\vertex[fill](2) at (2,0) {};
	\vertex[fill](3) at (3,0) {};
	\vertex[fill](5) at (4,0) {};
	\vertex[fill](8) at (5,0) {};
	\vertex[fill](9) at (6,0) {};
    \node() at (1,-0.5) {\footnotesize{$1$}};
    \node() at (2,-0.5) {\footnotesize{$2$}};
    \node() at (3,-0.5) {\footnotesize{$3$}};
    \node() at (4,-0.5) {\footnotesize{$5$}};
    \node() at (5,-0.5) {\footnotesize{$8$}};
    \node() at (6,-0.5) {\footnotesize{$9$}};

	\draw[] (1) to [bend left=60] (9);
    \draw[] (1) to [bend left=60] (2);
    \draw[] (3) to [bend left=60] (9);  
    \draw[] (3) to [bend left=60] (8);    
    \draw[] (5) to [bend left=60] (8);
\end{scope}
\end{tikzpicture}
    \caption{$M=x_{12}x_{19}x_{38}x_{39}x_{58}$ is a highest degree monomial in the reduction of $M_G$ of Example~\ref{ex:runningend}.  $M$ corresponds with the alternating non-crossing graph $G_M$ whose edges encode a facet in the triangulation of $\calP(G)$. The routes in the partially augmented graph $\widehat{G}_M$ form a facet in the triangulation of $\calF_{\hatG}$, and the arcs in the $(\IJ)$-tree $T$ form a facet in the triangulation of $\UIJ$.  Note that $\mathrm{prec}(T)=G_M$.
    }
    \label{fig:running}
\end{figure}

\section{\texorpdfstring{The $(\IJ)$-Tamari complex via the subdivision algebra}{}}
\label{sec.nuTamari}

In this section of the article, we apply Theorem~\ref{thm.main} to show that the $(\IJ)$-Tamari complex is encoded by the reduced form of the monomial $M_{\GIJ}$ obtained by performing reductions in the length reduction order.

The $(\IJ)$-Tamari complex $\mathscr{C}_{\IJ}$ is a simplicial complex introduced by Ceballos, Padrol, and Sarmiento~\cite{CPS19} in their study of $(\IJ)$-Tamari lattices.
They obtain $\mathscr{C}_{\IJ}$ via a triangulation of $\UIJ$, and show that the dual graph of the complex is the Hasse diagram of the $(\IJ)$-Tamari lattice. 
The $(\IJ)$-Tamari complex is defined in terms of objects known as $(\IJ)$-forests, which we now define. 

Recall that given a valid pair $(\IJ)$, the graph $\AIJ$ has vertex set $I\sqcup \overline{J}$ and edge set $\{(i,\overline{j}) \mid i \prec_{\IJ} \overline{j}, i\in I, \overline{j}\in \overline{J} \}$.
An {\em $(\IJ)$-forest} is a subgraph of $\AIJ$ that is {\em non-crossing}, meaning it does not contain two arcs $(i,\overline{j})$ and $(i',\overline{j}')$ satisfying $i \prec i' \prec \overline{j} \prec \overline{j}'$.
The $(\IJ)$-forests which contain the arc $(\min(I), \max(\overline{J}))$ and have no isolated nodes are known as {\em covering $(\IJ)$-forests}. 
An {\em $(\IJ)$-tree} is a maximal $(\IJ)$-forest, and has $|I|+|\overline{J}|-1$ edges.

\begin{definition}
The {\em $(\IJ)$-Tamari complex} $\mathscr{C}_{\IJ}$ is the flag simplicial complex on the set of $(\IJ)$-forests, whose minimal non-faces are pairs of crossing arcs.  
\end{definition}

The $(\IJ)$-Tamari complex can be realized as a triangulation of $\calU_{\IJ}$ known as the {\em $(\IJ)$-associahedral triangulation} $\mathscr{A}_{\IJ}$. 
This triangulation is characterized by the following rephrasing of a result by Ceballos, Padrol, and Sarmiento in \cite{CPS19}.

\begin{lemma}[{\cite[Lemma 4.3]{CPS19}}]
\label{lem.innerFaces}
Each interior simplex of the $(\IJ)$-Tamari complex in the $(\IJ)$-associahedral triangulation is given by
$$\mathrm{conv}\{(\be_i,\be_{\overline{j}}) \mid (i,\overline{j}) \text{ is an arc in }F\},$$
where $F$ is a covering $(\IJ)$-forest.
In particular, the facets of the $(\IJ)$-Tamari complex are given by the $(\IJ)$-trees.
\end{lemma}

Suppose $G$ is a graph whose vertex set $V(G) \subseteq [n]$.
Such a graph is said to be {\em alternating} if it does not have a pair of edges $(i,j)$ and $(j,k)$ satisfying $i<j<k$, and it is {\em non-crossing} if no pair of edges $(i,j)$ and $(i',j')$ satisfies $i<i'<j<j'$. 
Let $\calD_{\IJ}$ denote the set of all maximal alternating non-crossing graphs on the vertex set $I \cup \mathrm{prec}(\overline{J})$ whose edges $(i,j)$ satisfy $i\in I$, $j\in \mathrm{prec}(\overline{J})$.
Note that by maximality the graphs in $\calD_{\IJ}$ must be trees and therefore have $|I\cup \mathrm{prec}(\overline{J})| -1$ edges.

\begin{lemma}
\label{lem.IJtreeToAltGraphBijection}
Let $(\IJ)$ be a valid pair.
The set of $(\IJ)$-trees is in bijection with the set $\calD_{\IJ}$ of alternating non-crossing trees via the quotient operation $\mathrm{prec}$.
\end{lemma}
\begin{proof}
Recall from Section~\ref{sec.GIJ} that for any subgraph $H$ of $\AIJ$, the graph $\mathrm{prec}(H)$ is the quotient of $H$ where two vertices $i$ and $\overline{j}$ are identified if $i=\mathrm{prec}(\overline{j})$.

If $T$ is an $(\IJ)$-tree, then its arcs are non-crossing by definition, and so its quotient under $\mathrm{prec}$ remains a non-crossing tree on the vertex set $I \cup \mathrm{prec}(\overline{J})$.
Suppose for contradiction that $\mathrm{prec}(T)$ has a pair $(p,q)$, $(q,r)$ of non-alternating edges.  
We must have $p\in I$, $r\in \mathrm{prec}(\overline{J})$, and $q\in I \cap  \mathrm{prec}(\overline{J})$.  
Because all edges in $T$ are of the form $(i,\overline{j})$ for $i\in I$ and $\overline{j}\in \overline{J}$, then the pre-image of this pair of non-alternating edges $(p,q)$ and $(q,r)$ is a pair of crossing arcs in $T$, which is not allowed in an $(\IJ)$-tree.
Therefore, $\mathrm{prec}(T) \in \calD_{\IJ}$.

The inverse of $\mathrm{prec}$ is defined as follows.  
Given $D\in \calD_{\IJ}$, suppose the vertex $i\in I \cap \mathrm{prec}(\overline{J})$ is such that $i=\mathrm{prec}(\overline{j})$. 
If $(h,i)$ is an edge in $D$ then replace $(h,i)$ with the pair of non-crossing arcs $(h,\overline{j})$ and $(i,\overline{j})$.
Otherwise if $(i,k)$ is an edge in $D$ then replace $(i,k)$ with the pair of non-crossing arcs $(i, \overline{j})$ and $(i,\mathrm{prec}^{-1}(k))$.
Doing this for each vertex in $I \cap \mathrm{prec}(\overline{J})$ adds $|I \cap \mathrm{prec}(\overline{J})|$ edges to $D$, resulting in a non-crossing tree on the vertex set $I \sqcup \overline{J}$ with $|I \cup \mathrm{prec}(\overline{J})| - 1+ |I \cap \mathrm{prec}(\overline{J})| = |I|+|\overline{J}|-1$ edges, and so is an $(\IJ)$-tree.

Therefore, $\mathrm{prec}$ is a bijection between the set of $(\IJ)$-trees and the set $\calD_{\IJ}$.
\end{proof}

\begin{proposition}
\label{lem:anyorder}
The reduced forms of the monomial $M_{\GIJ}$ obtained by any reduction order in which the longest pair of edges is reduced at every reduction step, are equal. 
\end{proposition}

\begin{proof}
Equivalently, phrased in terms of graphs, it suffices to show that the set of full-dimensional leaves in such a reduction tree for $G=\GIJ$ is the set $\calD_{\IJ}$ of alternating non-crossing trees, as the full-dimensional leaves determine the remaining leaves in the reduction tree.

First, we note that a reduction of a longest pair of edges in a non-crossing graph $H$ preserves the non-crossing condition. 
Let $(i,j)$ and $(j,k)$ be the longest pair of edges at a vertex $j\in V(H)$.
Since no edge in $H$ crosses $(i,j)$ or $(j,k)$, and they are a longest pair at $j$, then the edge $(i,k)$ does not cross any edge in $H$. 
Thus each of the three graphs resulting from this reduction at $j$ is non-crossing. 

Let $R_G$ be a reduction tree for $G$ obtained by reducing a longest pair of edges at every reduction step. 
Since the graph $G$ is non-crossing by construction, then all graphs in $R_G$ are non-crossing.
Also, we note that a reduction at a vertex reduces the the number of non-alternating pairs of edges by one. 
Thus the leaves in the reduction tree are alternating non-crossing graphs.
In addition, $G$ is a tree by Lemmas~\ref{lem:nonoriented_acyclic} and~\ref{lem:Gconn}, so $|E(G)|= |I \cup \mathrm{prec}(\overline{J})|-1$.
Since simple reductions preserve the number of edges,
then the full-dimensional leaves in $R_G$ are alternating non-crossing graphs with $|I\cup \mathrm{prec}(\overline{J})|-1$ edges, and hence must be maximal. 
The full-dimensional leaves in $R_G$ correspond with the facets of a unimodular triangulation of $\calF_{\hatG}$, therefore they are enumerated by the normalized volume of $\calF_{\hatG}$, which by the integral equivalence $\varphi_1$ of Theorem~\ref{thm.main} is also the normalized volume of $\UIJ$. 
From Lemma~\ref{lem.innerFaces} it then follows that this volume is given by the number of $(\IJ)$-trees.
Since $(\IJ)$-trees are in bijection with the graphs in $\calD_{\IJ}$ by Lemma~\ref{lem.IJtreeToAltGraphBijection}, then the full-dimensional leaves of the reduction tree must comprise the whole set $\calD_{\IJ}$. 
Therefore, the full-dimensional leaves in $R_G$ is the set $\calD_{\IJ}$.
\end{proof}

The following corollary is now immediate.

\begin{corollary}
All reduced forms of $M_{\GIJ}$ obtained using a reduction order in which longest pairs are reduced first at each vertex give rise to the same triangulation of $\calU_{\IJ}$. 
\end{corollary}

Without loss of generality we may then use the length reduction order to refer to any reduction order in which longest pairs are reduced at every reduction step. 

\begin{theorem}
\label{thm.IJrealization}
The triangulation $\mathscr{T}_{len}$ of $\calF_{\hatG(I,\overline{J})}$ obtained by reducing the monomial $M_{\GIJ}$ in the length reduction order is a geometric realization of the $(\IJ)$-Tamari complex $\mathscr{C}_{\IJ}$.
\end{theorem}

\begin{proof}
The polytopes $\calF_{\GIJ}$ and $\UIJ$ have the same dimension, so it suffices to show that the integral equivalence $\varphi_1$ of Theorem~\ref{thm.main} maps the triangulation $\mathscr{T}_{len}$ of $\calF_{\GIJ}$ to the $(\IJ)$-associahedral triangulation $\mathscr{C}_{\IJ}$ of $\UIJ$. 

Let $R_G$ be the reduction tree for $G$ obtained by the length reduction order, and suppose $G_M\in \calD_{\IJ}$ is a full-dimensional leaf. 
By Lemma~\ref{lem.IJtreeToAltGraphBijection}, $G_M$ corresponds to the $(\IJ)$-tree $T=\mathrm{prec}^{-1}(G_M)$.

For every arc of the form $(i,\overline{j})$ in $T$, the route $(s, i, \mathrm{prec}(\overline{j}), t)$ is in the partially augmented graph $\hatG_M$.
Conversely, since $G_M$ is alternating, every route in $\hatG_M$ is of the form $(s,i,\mathrm{prec}(\overline{j}),t)$ with $i\leq \mathrm{prec}(\overline{j})$, and $(i,\overline{j})$ is an arc in $T$. 
Therefore, there is a bijection between the arcs in $T$ and the routes in $\hatG_M$.

The integral equivalence $\varphi_1$ takes the vertex corresponding to the route $(s,i,\mathrm{prec}(\overline{j}),t)$ in $\widehat{G}_M$ to the vertex $(\be_i,\be_{\overline{j}} )$ corresponding to the arc $(i,\overline{j})$ in $T$, and therefore, sends the facet of $\mathscr{T}_{len}$ encoded by $\hatG_M$ to the facet of $\mathscr{C}_{\IJ}$ encoded by $T$.
\end{proof}

\begin{example}
Again consider the valid pair $I=\{1,2,3,5,9 \}$, $\overline{J}=\{\overline{2}, \overline{7}, \overline{8}, \overline{9}\}$ and let $G=\GIJ$.
$M=x_{12}x_{19}x_{38}x_{39}x_{58}$ is a highest degree monomial in the reduced form of $M_G$ obtained by the length reduction order, and its corresponding alternating non-crossing tree $G_M$ is pictured in Figure~\ref{fig:running}.
Under the integral equivalence $\varphi_1$, the eight routes in $\hatG_M$ correspond to the eight arcs in the $(\IJ)$-tree $T$, and respectively they encode the vertices in a facet of the triangulation $\mathscr{T}_{len}$ of $\calF_{\hatG}$, and the vertices in a facet of the $(\IJ)$-Tamari complex $\mathscr{C}_{\IJ}$.
\end{example}

\begin{remark}
In fact, the set of covering $(\IJ)$-forests is in bijection with the set of graphs in the leaves of a reduction tree $R_G$ obtained by any reduction order in which the longest pair of edges is reduced at every reduction.  
The benefit of the $(\IJ)$-Tamari complex viewpoint is that the set of covering $(\IJ)$-forests is easily characterizable.
The covering $(\IJ)$-forests (respectively graphs in the leaves of $R_G$) index the interior faces of $\mathscr{C}_{\IJ}$ (respectively interior faces of the triangulation $\mathscr{T}_{len}$).
\end{remark}


From an $(\IJ)$ pair we obtain a lattice path $\overline{\nu}$ by reading the elements in increasing order, associating elements in $I$ with $E$ steps and elements in $\overline{J}$ with $N$ steps. 
Deleting the initial $E$ step and terminal $N$ step from $\overline{\nu}$ then gives a lattice path $\nu$.  
For example, the pair $I=\{1,2,3,5,9\}$, $\overline{J}=\{\overline{2}, \overline{7}, \overline{8}, \overline{9}\}$ gives rise to $\nu = ENEENNE$.

In general, the same $\nu$ can be obtained from different pairs $(\IJ)$. However, the $(\IJ)$ pairs giving rise to the same $\nu$ all induce combinatorially equivalent $(\IJ)$-Tamari complexes. 
For this reason, the $(\IJ)$-Tamari complex can simply be referred to as the $\nu$-Tamari complex $\calC_\nu$.

Several enumerative properties of the $\nu$-Tamari complex can be easily described using lattice paths. 
A {\em $\nu$-Dyck path} is a lattice paths using steps $E=(1,0)$ and $N=(0,1)$ staying weakly above the path $\nu$ that shares its end points.
The number of $\nu$-Dyck paths is the {\em $\nu$-Catalan number} $\Cat(\nu)$. 
The {\em $\nu$-Narayana number} $\Nar_\nu(i)$ is the number of $\nu$-Dyck paths with $i$ valleys (consecutive $EN$ pairs).
Similar to a $\nu$-Dyck path, a {\em $\nu$-Schr\"oder path} is a lattice path using steps $E=(1,0)$, $N=(0,1)$, and $D=(1,1)$ staying weakly above the path $\nu$ that shares its end points. The number of $\nu$-Schr\"oder paths with $i$ diagonal steps is the $\nu$-Schr\"oder number $\mathrm{Sch}_\nu(i)$. 

Ceballos, Padrol, and Sarmiento~\cite{CPS19} showed that the number of facets in $\calC_\nu$ (and the volume of $\calU_{\IJ}$) is given by $\Cat(\nu)$.
They further showed that $\calC_\nu$ generalizes the boundary complex of the simplicial associahedron, and that it is be shellable with its $h$-vector given by $\nu$-Narayana numbers. 

Define the {\em $\nu$-Narayana polynomial} to be $N_\nu (x) = \sum_{i\geq 0} \Nar_{\nu} (i)x^i$ and the {\em $\nu$-Schr\"oder polynomial} to be $S_\nu (x) = \sum_{i\geq 0} \mathrm{Sch}_{\nu} (i)$.
We have the following corollary of Theorem~\ref{thm.IJrealization}.

\begin{corollary}
Let $p$ be any reduced form of $M_{\GIJ}$. 
Then $p(x_{ij} = 1,\beta)$ is the $\nu$-Schr\"oder polynomial $S_\nu(\beta)$, whose coefficients $\mathrm{Sch}_\nu(i)$ count the inner $i$-faces of $\calC_\nu$. In addition, the $h$-polynomial of $\calC_\nu$ is given by $\rho(x_{ij}=1, \beta-1)$, which is the $\nu$-Narayana polynomial $N_\nu(\beta)$.  
\end{corollary}

\begin{proof}
M\'esz\'aros \cite[Theorem 8]{Mes15} showed that the $h$-polynomial of a triangulated flow polytope $\calF_{\widetilde{G}}$ can be obtained from any reduced form $p$ of $M_G$ as a polynomial in $\beta$. 
In particular the $h$-polynomial is the polynomial $p(x_{ij}=1,\beta-1)$, obtained by first substituting $x_{ij} = 1$ for all $x_{ij}$ in $p$, and then computing $p(\beta -1)$. 
Since we know from \cite[Theorem 4.6]{CPS15} that the $h$-vector of the $\nu$-Tamari complex is given by the $\nu$-Narayana numbers, the $h$-polynomial is $N_\nu(x)$.  
The coefficient of $\beta^i$ in the shifted $h$-polynomial $p(x_{ij}=1, \beta)$ counts the number of inner $i$-faces $\calC_{\IJ}$. 
It was shown in \cite[Proposition 2.3]{BY21} that $p(x_{ij}=1, \beta) = N_\nu(x+1)$ is the $\nu$-Schr\"oder polynomial $S_\nu(x)$.
\end{proof}

In the preprint~\cite{Bel22}, the first author considers triangulating flow polytopes of graphs similar to $\widehat{G}(I,\overline{J})$ where some edges are bi-directional. 
As a consequence, reduced forms of certain polynomials in a generalized subdivision algebra are shown to encode triangulations of a product of two simplices. A particular reduction order analogous to the length reduction order gives rise to a triangulation dual to the $(\IJ)$-cyclohedron of \cite{CPS19}.

\DeclareRobustCommand{\VON}[3]{#3}
\bibliographystyle{plain}
\bibliography{nuTamari}

\end{document}